\documentclass[11pt]{amsart}
\usepackage{graphicx}
\usepackage{amscd}
\usepackage{amsmath}
\usepackage{amsfonts}
\usepackage{amssymb}
\usepackage[foot]{amsaddr}
\usepackage{bbm}
\usepackage{setspace}
\usepackage{enumerate}         
\usepackage{color}
\usepackage{url}
\usepackage{amsthm}
\usepackage{hyperref}
\usepackage{bm}
\usepackage{xy}
\usepackage{color}
\usepackage{ulem}
\usepackage{cancel}
\usepackage{subfig}
\usepackage{bbm}
\usepackage{etoolbox}
\usepackage[pagewise]{lineno}
\usepackage{mathrsfs}
\usepackage{algpseudocode}
\usepackage{algorithmicx}
\usepackage[ruled,vlined]{algorithm2e}
\usepackage{neuralnetwork}
\allowdisplaybreaks[4]
\usepackage{booktabs}

\definecolor{teal}{RGB}{0, 128, 128}
\definecolor{violet}{RGB}{138, 43, 226}
\definecolor{darkmagenta}{RGB}{139, 0, 139}
\definecolor{navyblue}{RGB}{0, 0, 128}
\definecolor{darkgray}{RGB}{64, 64, 64}

\usepackage{natbib}

\usepackage{geometry}
\usepackage{verbatim}

\usepackage{tikz}
\usetikzlibrary{positioning,shapes,shadows,arrows}
\tikzstyle{abstract}=[rectangle, draw=black, rounded corners, fill=blue!40, drop shadow,
        text centered, anchor=north, text=white, text width=2.5cm]
\tikzstyle{comment}=[rectangle, draw=black, rounded corners, fill=green, drop shadow,
        text centered, anchor=north, text=white, text width=2cm]
\tikzstyle{myarrow}=[->, >=open triangle 90, thick]
\tikzstyle{line}=[-, thick]

\theoremstyle{plain}
\newtheorem{theorem}{Theorem}[section]

\newtheorem{lemma}[theorem]{Lemma}

\theoremstyle{remark}
\newtheorem{remark}[theorem]{Remark}

\numberwithin{equation}{section}

\newcommand{\rsto}{]\!\kern-1.8pt ]}
\newcommand{\lsto}{[\!\kern-1.7pt [}

\numberwithin{equation}{section}

\renewcommand{\emph}[1]{\textit{#1}}

\newcommand{\FF}{\mathbb{F}}

\newcommand{\RR}{\mathbb{R}}
\newcommand{\QQ}{\mathbb{Q}}

\newcommand{\PP}{\mathbb{P}}
\newcommand{\NN}{\mathbb{N}}
\newcommand{\EE}{\mathbb{E}}

\newcommand{\cU}{\mathcal{U}}

\newcommand{\cF}{\mathcal{F}}

\newcommand{\cM}{\mathcal{M}}

\newcommand{\tr}{\mathop{\mathrm{Tr}}}

\newcommand{\eps}{\epsilon}

\newcommand{\de}{\mathrm{d}} 
\renewcommand{\cite}{\citet}

\DeclareMathOperator*{\argmin}{arg\,min}

\makeatletter
\patchcmd{\@maketitle}
  {\ifx\@empty\@dedicatory}
  {\ifx\@empty\@date \else {\vskip3ex \centering\footnotesize\@date\par\vskip1ex}\fi
   \ifx\@empty\@dedicatory}
  {}{}
\patchcmd{\@adminfootnotes}
  {\ifx\@empty\@date\else \@footnotetext{\@setdate}\fi}
  {}{}{}
\newcommand{\subjclassname@JEL}{JEL Classification}
\makeatother

\begin{document}

\title{A Deep solver for BSDEs with jumps}

\author{Kristoffer Andersson}
\address[Kristoffer Andersson, Alessandro Gnoatto, Athena Picarelli]{University of Verona, Department of Economics, \newline
\indent via Cantarane 24, 37129 Verona, Italy}
\email[Kristoffer Andersson]{kristofferherbert.andersson@univr.it}

\author{Alessandro Gnoatto}
\email[Alessandro Gnoatto]{alessandro.gnoatto@univr.it}

\author{Marco Patacca}
\address[Marco Patacca]{University of Perugia, Department of Economics, \newline
\indent Via Alessandro Pascoli 20, 06123 Perugia, Italy}
\email[Marco Patacca]{marco.patacca@unipg.it}%

\author{Athena Picarelli}
\email[Athena Picarelli]{athena.picarelli@univr.it}%

\begin{abstract}
The aim of this work is to propose an extension of the deep solver by Han, Jentzen, E (2018) to the case of forward backward stochastic differential equations (FBSDEs) with jumps. As in the aforementioned solver, starting from a discretized version of the FBSDE and parametrizing the (high dimensional) control processes by means of a family of artificial neural networks (ANNs), the FBSDE is viewed as a model-based reinforcement learning problem and the ANN parameters are fitted so as to minimize a prescribed loss function. We take into account both finite and infinite jump activity by introducing, in the latter case, an approximation with finitely many jumps of the forward process. {We successfully apply our algorithm to option pricing problems in low and high dimension and discuss the applicability in the context of counterparty credit risk.}
\end{abstract}

\keywords{BSDE with jumps, Deep BSDE Solver, Neural Networks, PIDE, pricing of derivatives, risk management}
\subjclass[2010]{93E20, 65M75, 68T07, 60H10. \textit{JEL Classification} C02, C63}

\date{\today}

\maketitle

\section{Introduction}\label{intro}
Forward backward stochastic differential equations (FBSDEs) have become a popular tool in several application domains. One reason is that, due to the Feynman-Kac theorem, they essentially represent the stochastic counterpart of partial differential equations (PDEs): such equations naturally arise when modelling phenomena in diverse fields such as physics, engineering and finance. In finance, for example, PDEs/FBSDEs emerge in the context of the problem of pricing a contingent claim, when the underlying security, i.e. the state variable, is modelled by means of a diffusion process. Similar remarks apply when we consider state variables following a jump diffusion process: in this case a further integral/non local term appears in the Cauchy problem, which is then referred to as a partial integro-differential equation (PIDE). BSDEs in this case are in turn generalized by the introduction of a jump term driven by a Poisson random measure.

When a closed-form solution to the PIDE is not available and the dimension of the vector of state variables is low, several approaches for the numerical solution are availble, the most famous being finite difference and finite element methods. For further details, we refer the reader to \cite{Cont2003}, \cite{Hilber2013}. However, the application of such standard numerical techniques becomes increasingly difficult as the dimension of the state space increases: the tightness of error bounds may be negatively affected or, more simply, the required computational time might increase significantly. Such phenomena are often referred to as the \textit{curse of dimension}.

The above mentioned problems in a high dimensional setting provide the motivation for the recent surge in interest in machine-learning based methods to solve PDEs/BSDEs, where artificial neural networks (ANNs) are employed in order to parametrize e.g. the function satisfying the PDE and/or its gradient. From a mathematical perspective, ANNs are multiple nested compositions of
relatively simple multivariate functions.
ANNs can be graphically represented by logical maps with a structure that loosely resembles the one of the
human brain, where each neuron (corresponding to the application of a simple function on a multi-dimensional vector) is linked to a multitude of
neighboring neurons grouped in sequential layers. 
The term deep neural networks refers to ANNs with several interconnected layers.
One remarkable property of ANNs is given in the `Universal Approximation Theorem', which essentially states that any continuous function in any dimension can be represented to arbitrary accuracy by means of an ANN, and has been proven in different versions,
starting from the remarkable insight of Kolmogorov's Representation Theorem in \cite{kolmo56} and the seminal works of \cite{cybenko89} and \cite{hornik91}.
In a nutshell, this result states that any continuous function in any dimension can be represented to arbitrary accuracy by means of an ANN. Several authors have shown, under different assumption and in different settings, that ANNs can overcome the curse of dimension when approximating the solutions of PDEs, see e.g.  \cite{jentzen2018proof}, \cite{reizha19}, \cite{hutzenthaler2020overcoming}.

Recently, several authors have proposed different numerical schemes for PDEs/BSDEs based on the Feynman-Kac theorem in absence of jumps. Seminal works in this sense are \cite{ehanjen17} and \cite{Han2018}, together with the convergence study in \cite{HanLon18}. In these works, the BSDE is discretized forward in time via a standard Euler scheme. The initial condition and the controls of the BSDE, at each point in time, are then parametrized by a family of ANNs. The parameters of the ANNs are optimized by minimizing the expected square distance between the known terminal condition and the terminal value of the discretized BSDE. An alternative route has been followed in \cite{hpw2021}, involving a backward recursion based on the dynamic programming principle. 
Another alternative is the so-called \textit{deep splitting method}, first proposed in \cite{bbcjn2021}: in this case the differential operator of a parabolic PDE is split into a linear and a non-linear part. The non-unique split is chosen is such a way that the non-linear part becomes small. The PDE is then solved iteratively over small time intervals. Such numerical solution involves the recursive computation of conditional expectations which are approximated by means of ANNs. We also mention, among others, the \textit{deep Galerkin method} proposed in \cite{SIRIGNANO2018}: this method is not based on the link between FBSDE and PDEs instead, in their approach, the ANN is trained in order to satisfy the differential operator, the initial condition, and the boundary conditions. The above mentioned papers have been generalized in a multitude of directions and we refer to \cite{bbcjn2021} for an extensive literature review.

{Concerning cases involving jumps, we mention the work by \cite{BARIGOU2022113922}, focusing on actuarial applications and the  recent preprint \cite{georgoulis2024deep}}.

The aforementioned contribution \cite{hpw2021} has been extended to the jump diffusion case in \cite{castro2021}, which also proposes a convergence analysis. {More recently, the backward dynamic programming approach has been considered also in \cite{bcn23}, where reflected BSDEs are also considered.} A generalization of \cite{bbcjn2021} is presented in \cite{vk2021} with a convergence analysis provided in \cite{frey2025convergence}. The deep Galerkin method instead is extended in \cite{al2022extensions}. { We point out that in these works only finite jump activity is taken into account.}

In this paper, we generalize the algorithm of \cite{ehanjen17} and \cite{Han2018} to the jump diffusion setting { taking into account the possible infinite jump activity}. Following their reasoning, we discretize the FBSDE with jumps in a forward loop with respect to the time dimension. Next, we introduce deep learning approximations, in the form of two families of ANNs with different tasks: the first family approximates the value process of the BSDE, i.e. the solution of the PIDE, whereas the second is introduced in order to approximate the possibly high dimensional integral with respect to the L\'evy measure that appears in the discretized BSDE. Our algorithm is first formulated for processes where the jump term exhibits finite activity. The first family of networks is trained by minimizing the expected square distance between the terminal value of the parametrized BSDE and the known terminal condition, in line with \cite{ehanjen17} and \cite{Han2018}. To train the second family of ANNs we exploit the martingale property of compensated Poisson integrals and the $L^2$-minimality of conditional expectations, which results in a second penalty term to be minimized.  We also propose an extension for the infinite activity/infinite variation of jumps case: we rely on the literature on the approximation of paths of infinite activity L\'evy processes by means of compound Poisson processes, (see e.g. \cite{ar2001} among others) and we also introduce a perturbation of the diffusion coefficient to approximate jumps smaller than a pre-specified threshold. {We provide a full error analysis of the approximation of infinite activity FBSDEs by means of finite activity FBSDEs. The full study of the convergence of the algorithm beyond the scope of the present paper and is the subject of future work.}  We point out that, during the review process of this paper,  the preprint \cite{abdw24} appeared, introducing new algorithms and presenting an insightful comparison between local and global solvers for (possibly coupled) FBSDEs with jumps.\color{black}
 
 The paper is organized as follows. In Section \ref{Sect:preliminaries}, we present our main assumptions and notations, whereas Section \ref{sec:algorithm} is devoted to the presentation of our algorithm both in the finite and the infinite activity case. Section \ref{sec:errorEstimates} presents some a priori error estimates that study the distance between the solution of the true FBSDE with jumps of infinite activity and the one of the approximating FBSDE driven by a compound Poisson process. Finally Section \ref{sec:numerics} presents some numerical experiments both in the finite and infinite activity setting with a focus on financial modeling: among others, we show that our methodology can be used to treat a PIDE in dimension $100$ related to the pricing of a basket option. {We also highlight the applicability of our methodology in the context of counterparty credit risk along the lines of \cite{gpr2020}}

\section{Setting and preliminaries}\label{Sect:preliminaries}

We fix a time horizon $T\in\RR_+$ with $T<\infty$. Let $\left(\Omega,\cF,\FF,\QQ\right)$ be a filtered probability space, where the filtration $\FF=\left(\mathcal{F}_t\right)_{t\in[0,T]}$ satisfies the usual assumptions. All semimartingales introduced in the following are c\`adl\`ag. 

We suppose that the filtered probability space supports an $\RR^d$ valued standard Brownian motion $W=(W_t)_{t\in[0,T]}$, together with a Poisson random measure $N$ with associated L\'evy measure $\nu$, such that
\begin{align}
\label{eq:assumpLevyMeasure}
\nu(\{0\}) = 0, \quad \int_{\RR^d}1 \wedge |z|^2 \nu(\de z)<\infty,\quad \text{ and }\int_{|z|\geq 1}|z|^2 \nu(\de z)<\infty,
\end{align}
allowing us to introduce the compensated random measure 
\begin{equation}\label{eq:compensated}
\tilde N(\mathrm d t,\mathrm d z) : = N(\mathrm d t,\mathrm d z) - \nu(\mathrm d z) \mathrm d t.
\end{equation}
Let us introduce the following spaces
\begin{itemize}
\item $\mathbb{L}_{\mathcal F}^{2}\left(\mathbb{R}^{d}\right)$, the space of all $\mathcal{F}$-measurable random variables $X: \Omega \rightarrow \mathbb{R}^{d}$ satisfying $$\|X\|^{2}:=\mathbb{E}\left(|X|^{2}\right)<+\infty;$$
\item $\mathbb{H}_{T}^{2}\left(\mathbb{R}^{d}\right)$, the space of all predictable process $\phi: \Omega \times[0, T] \rightarrow \mathbb{R}^{d}$ such that $$\|\phi\|^{2}_{\mathbb{H}^{2}}:=\mathbb{E}\left[\int_{0}^{T}\left|\phi_{t}\right|^{2} \mathrm{~d} t\right]<+\infty.$$
\end{itemize}
As previously mentioned, the filtration $\FF$ supports the Brownian motion $W$ and the Poisson random measure $N$.

%
%
%
Next, we introduce the following spaces that are needed to define the concept of solution to FBSDEs.

\begin{itemize}
\item $\mathbb{H}_{T, N}^{2}\left(\mathbb{R}^{d}\right)$ is the space of all predictable process $\phi: \Omega \times[0, T] \times \mathbb{R}^{d} \rightarrow \mathbb{R}^{d}$ satisfying $$\|\phi\|^2_{\mathbb{H}_{N}^{2}}:=\mathbb{E}\left[\int_{0}^{T} \int_{\mathbb{R}^{d}}\left|\phi_{t}(z)\right|^{2} \nu(\de z) \de t\right]<+\infty,$$ where we integrate with respect to the predictable compensator of the Poisson random measure $N$.
\item $\mathbb{S}_{T}^{2}\left(\mathbb{R}^{d}\right)$ is the space of $\mathbb{F}$-adapted càdlàg processes $\phi: \Omega \times[0, T] \rightarrow \mathbb{R}^{d}$ satisfying $$\|\phi\|_{\mathbb S^2} :=\mathbb{E}\left[\sup _{0\leq t \leq  T}\left|\phi_{t}\right|^{2}\right]<+\infty.$$
\end{itemize}

Let us consider a SDE in the following general form:
\begin{align}
X _ { s } &  = x + \int _ { t } ^ { s } b \left( X _ { r- } \right) \mathrm d r + \int _ { t } ^ { s } \sigma \left( X _ { r- } \right)^\top \mathrm d W _ { r }  +  
\int^s_t \int_{\RR^d} \Gamma \left( X _ { r- } , z\right) \tilde N(\mathrm d r, \mathrm d z),\, \; t\in [0,T],\, x \in \mathbb { R } ^ { d }.
 \label{eq:forward}
\end{align}

The vector fields $b\colon \mathbb{R}^d \to \mathbb{R}^d$, $\sigma\colon\mathbb{R}^d\to \mathbb{R}^{d\times d}$, $\Gamma\colon \RR^d \times \RR^d \to \RR^d$ are measurable functions satisfying the following assumptions.
\begin{itemize}
\item[(A1)] The functions $b$ and $\sigma$ are Lipschitz continuous;
\item[(A2)] The function $\Gamma$ satisfies, for some constant $K>0$
\begin{align*}
\begin{aligned}
&\int_{|z|<1}\left|\Gamma\left(x, z\right)-\Gamma\left(x^\prime, z\right)\right|^2 \nu(\de z) \leq K\left|x-x^\prime\right|^2\\
&\int_{|z|<1}|\Gamma(x, z)|^2 \nu(\de z) \leq K\left(1+|x|^2\right)
\end{aligned}
\end{align*}
\end{itemize}

Under such assumptions the SDE \eqref{eq:forward} admits a unique strong solution with initial time $t$ and initial condition $x\in \RR^d$ that we denote by $X^{t,x}=(X^{t,x}_s)_{t\in[t,T]}$. In our setting the $\sigma$-algebra $\mathcal F_t$ is then the trivial one.

\begin{theorem} Let assumptions (A1)-(A2) be satisfied. The following holds:
\begin{enumerate}
\item for each $(t, x) \in[0, T] \times \mathbb{R}^{d}$ exists a unique adapted, c\`adl\`ag solution $X^{t, x}:=$ $(X_{s}^{t, x})_{t \leq s \leq T}$ to \eqref{eq:forward};
\item the solution $X^{t, x}$ is a homogeneous Markov process.
\end{enumerate}
\end{theorem}
\begin{proof}
The claims follow from Theorem 6.2.9. and Theorem 6.4.6 in \cite{appl09} respectively.
\end{proof}

{
\begin{remark}
We assume the coefficients in \eqref{eq:forward} are time independent for simplicity. All the results of the paper may be extended to the time dependent case.
 \end{remark}
}
Let us then introduce the following BSDE:
\begin{equation}\label{eq:backward}
\begin{aligned}
 Y_{t}    = &  g (X_T) +\int_{t}^{T} f \left(s, X_{r-}, Y_{r-}, Z_{r}, \int_{\RR^d} U_r(z) \nu(\mathrm d z) \right) \mathrm d r- \int_{t}^{T} Z_{r}^\top \mathrm d W_{r} \\
& - \int^T_t \int_{\RR^d} U_r(z) \tilde N(\mathrm d r, \mathrm d z). 
\end{aligned}
\end{equation}

\begin{remark}
A more general formulation of the BSDE \eqref{eq:backward} may involve a driver of the form $f:[0, T] \times \mathbb{R}^d \times \mathbb{R} \times \mathbb{R}^d \times L^2_{\nu}(\RR^d) \rightarrow \mathbb{R} $, where $L^2_{\nu}(\RR^d)$ is the space of functions $\phi$ such that $\int_{\mathbb{R}^d}|\phi(z)|^2\nu(\de z)<\infty$,  i.e. {the driver in \eqref{eq:backward} may be replaced by a generic functional} $f \left(s, X_{r-}, Y_{r-}, Z_{r}, U_r(\cdot) \right)$. We introduce instead a dependence on the integral of $U$ with respect to the L\'evy measure, which is often found in applications, see \cite{bbp1997} and \cite{Delongbook}.
\end{remark}

We consider the following assumptions:
\begin{enumerate}
\item[(A3)] the function $f:[0, T] \times \mathbb{R}^{d} \times \mathbb{R} \times \mathbb{R}^{d} \times \mathbb{R}^{d} \rightarrow \mathbb{R}$ is Lipschitz continuous with respect to the state variables, uniformly in $t$, i.e. there exists $K>0$ such that 
$$
\begin{aligned}
&\left|f(t, x, y, z, u)-f\left(t, x^{\prime}, y^{\prime}, z^{\prime}, u^{\prime}\right)\right| \\
&\quad \leq K\left(\left|x-x^{\prime}\right|+\left|y-y^{\prime}\right|+\left|z-z^{\prime}\right|+\left|u-u^{\prime}\right|\right)
\end{aligned}
$$
for all $(t, x, y, z, u),\left(t, x^{\prime}, y^{\prime}, z^{\prime}, u^{\prime}\right) \in[0, T] \times \mathbb{R}^{d} \times \mathbb{R} \times \mathbb{R}^{d} \times \mathbb{R}^{d}$;
\item[(A4)]  the function $g: \mathbb{R}^{d} \rightarrow \mathbb{R}$ is Lipschitz continuous, i.e. there exists $K>0$ such that 
$$
\left|g(x)-g\left(x^{\prime}\right)\right| \leq K\left|x-x^{\prime}\right|,
$$
for all $x, x^{\prime} \in \mathbb{R}^{d}$.
\end{enumerate}

The above conditions allow us to recall the following existence and uniqueness theorem. We use again superscripts $t,x$ to stress the dependence on the initial condition $x$ at time $t$.

\begin{theorem}\label{th:existUniq} [Theorem 4.1.3, \cite{Delongbook}] Under assumptions (A1)-(A4) there exists a unique solution $(X^{t, x},Y^{t, x},Z^{t, x},U^{t, x}) \in \mathbb{S}_{T}^{2}\left(\mathbb{R}^{d}\right) \times$ $\mathbb{S}_{T}^{2}(\mathbb{R}) \times \mathbb{H}_{T}^{2}\left(\mathbb{R}^{d}\right) \times \mathbb{H}_{T, N}^{2}\left({\mathbb{R}}\right)$ to the FBSDE \eqref{eq:forward}-\eqref{eq:backward}.
\end{theorem}

To alleviate notations, hereafter we omit when possible the dependency on the initial condition $(t,x)$ of the processes $(X^{t,x}_\cdot,Y^{t,x}_\cdot, Z^{t,x}_\cdot,U^{t,x}_\cdot)$.

Our objective is to numerically solve FBSDEs with jumps, possibly in a high dimensional setting, by means of deep learning techniques. It is well known that FBSDEs with jumps are strongly related to PIDEs, so our solver offers in turn a numerical approach to the solution of high dimentional PIDEs. The main theoretical tool that underpins this link is the Feynman-Kac theorem, that links PIDEs and FBSDEs in a Markovian setting. For the reader's convenience, let us recall the version of the Feynman-Kac theorem that is relevant for our purposes. We are interested in finding a function $u\in C^{1,2}([0,T]\times \mathbb{R}^d,\mathbb{R})$ that satisfies the PIDE
\begin{align}
\label{eq:PIDE1}
\begin{aligned}
&-u_{t}(t, x)-\mathscr{L} u(t, x) \\
&\quad-f\left(t, x, u(t, x), D_x u(t, x) \sigma(x), \mathscr{J} u(t, x)\right)=0, \quad(t, x) \in[0, T) \times \mathbb{R}^d, \\
&u(T, x)=g(x), \quad x \in \mathbb{R}^d,
\end{aligned}
\end{align}
where
\begin{align}
\label{eq:PIDE2}
\begin{aligned}
\mathscr{L} u(t, x)=& \left\langle \mu(x), D_x u(t, x)\right\rangle+\frac{1}{2} \left\langle\sigma(x) D^2_{x} u(t, x), \sigma(x)\right\rangle \\
&+\int_{\mathbb{R}^d}\left(u(t, x+\Gamma(x, z))-u(t, x)-\left\langle\Gamma(x, z), D_x u(t, x)\right\rangle\right) \nu(d z), \\
\mathscr{J} u(t, x)=& \int_{\mathbb{R}^d}(u(t, x+\Gamma(x, z))-u(t, x))\nu(d z).
\end{aligned}
\end{align}
The link between the FBSDE \eqref{eq:forward}-\eqref{eq:backward} and the PIDE \eqref{eq:PIDE1}-\eqref{eq:PIDE2} is established via the following :

\begin{theorem}[Theorem 4.2.1, \cite{Delongbook}] \label{th:FeynmanKac}
Let assumptions (A1)-(A4) be satisfied and  let  $u\in C^{1,2}([0,T]\times \mathbb{R}^d,\mathbb{R})$ satisfy the PIDE \eqref{eq:PIDE1}-\eqref{eq:PIDE2} and the linear growth conditions
\begin{align}
|u(t, x)| \leq K(1+|x|), \quad\left|u_{x}(t, x)\right| \leq K(1+|x|), \quad\forall\, (t, x) \in[0, T] \times \mathbb{R}^d,
\end{align}
then 
\begin{align}
\begin{aligned}
Y^{t, x}_s=& u(s, {X}^{t, x}_s), \quad t \leq s \leq T, \\
Z^{t, x}_s=& \sigma({X}^{t, x}_{s-})^\top D_x u(s, {X}^{t, x}_{s-}), \quad t \leq s \leq T, \\
U^{t, x}_s(z)=& u(s, {X}^{t, x}_{s-}+\Gamma({X}^{t, x}_{s-}, z)) -u(s, {X}^{t, x}_{s-}), \quad t \leq s \leq T, z \in \mathbb{R}^d. 
\end{aligned}
\end{align}
\end{theorem}

The above formulation of FBSDEs is intrinsically linked to the following stochastic optimal control problem:
\begin{align}
\label{eq:stochControlProblem}
&\underset{\begin{array}{c}y\in\RR, \\Z\in \mathbb{H}_{T}^{2}\left(\mathbb{R}^{d}\right),\\U\in \mathbb{H}_{T,N}^{2}\left(\mathbb{R}\right)\end{array}}{\text{minimize}}\mathbb{E}\left[\left(g\left(X^{t, x}_{T}\right)-Y^{y,Z,U}_{T}\right)^{2}\right], \nonumber\\
& \text{subject to: }\\
&\begin{cases}
X ^{t, x}_ { s} =  x + \int _ { t } ^ { s } b ( X^{t, x} _ { r- }) \mathrm d r + \int _ { t } ^ { s } \sigma ( X ^{t, x}_ { r- })^\top \mathrm d W _ { r }  +  
\int^s_t \int_{\RR^d} \Gamma ( X^{t, x} _ { r- } , z	) \tilde N(\mathrm d r, \mathrm d z), \\
Y^{y,Z,U}_{s}  =  y    -\int_{t}^{s} f (r, X^{t,x}_{r-}, Y^{y,Z,U}_{r-}, Z_{r}, \int_{\RR^d} U_r(z) \nu(\mathrm d z) ) \mathrm d r+ \int_{t}^{s} (Z_{r})^\top \mathrm d W_{r} \nonumber \\
\quad  \quad\qquad + \int^s_t \int_{\RR^d} U_r(z) \tilde N(\mathrm d r, \mathrm d z), \hfill s\in[t,T]. 
\end{cases}
\end{align}
\begin{lemma}
Under assumptions (A1)-(A4), the minimum in problem \eqref{eq:stochControlProblem} is achieved and the corresponding minimizer $(y^\star,Z^\star,U^\star)$ is the unique solution to the FBSDE \eqref{eq:forward}-\eqref{eq:backward}. In particular, one has
$$
\mathbb{E}\left[\left(g\left(X^{t, x}_{T}\right)-Y^{y^\star,Z^\star,U^\star}_{T}\right)^{2}\right]=0
$$
and 
\begin{align*}
\begin{aligned}
y^\star &=Y^{t, x}_t= u\left(t, x\right), \\
Z^\star_s &=Z^{t, x}_s= \sigma({X}^{t, x}_{s-})^\top D_x u(s, {X}^{t, x}_{s-}), \quad t \leq s \leq T, \\
U^\star_s(z)& =U^{t, x}_s(z)= u(s, {X}^{t, x}_{s-}+\Gamma({X}^{t, x}_{s-}, z)) -u(s, {X}^{t, x}_{s-}), \quad t \leq s \leq T, z \in \mathbb{R}^d.
\end{aligned}
\end{align*}
\end{lemma}

\begin{proof} Thanks to assumptions (A1)-(A4) the FBSDE \eqref{eq:forward}-\eqref{eq:backward} has a unique solution $({X}^{t, x},{Y}^{t, x},{Z}^{t, x},{U}^{t, x})$ as stated in Theorem \ref{th:existUniq}. We consider now $({Y}^{t, x}_t,{Z}^{t, x},{U}^{t, x})$ as the control of the minimization problem \eqref{eq:stochControlProblem}, observing that the dynamics constraint is satisfied. Since we have ${Y}^{t, x}_T=g({X}^{t, x}_T)$ $\PP$-a.s. it means that 
\begin{align*}
\mathbb{E}\left[\left(g\left(X^{t, x}_{T}\right)-Y^{{Y}^{t, x}_t,{Z}^{t, x},{U}^{t, x}}_{T}\right)^{2}\right]=0
\end{align*}
and the minimum is then achieved at $(y^\star,Z^\star,U^\star)=({Y}^{t, x}_t,{Z}^{t, x},{U}^{t, x})$. Moreover, any other minimizer should be a solution of the FBSDE, but being such a solution unique we can conclude  by Theorem \ref{th:FeynmanKac} that the last statement of the lemma holds.
\end{proof}

\subsection{The Deep BSDE solver by E-Han-Jentzen}
We provide an overview of the algorithm proposed in \cite{ehanjen17}. The main idea of their approach is to consider, in a setting without jumps, a discretized version of the stochastic control problem in \eqref{eq:stochControlProblem} and then approximate, at each time step,  the control process by using an artificial neural network (ANN). First of all, let us observe that in absence of jumps the dynamics constraints in \eqref{eq:stochControlProblem} becomes
\begin{equation}\label{dyn_no_jumps}
\begin{cases}
X ^{t, x}_ { s} =  x + \int _ { t } ^ { s } b ( X^{t, x} _ { r- }) \mathrm d r + \int _ { t } ^ { s } \sigma ( X ^{t, x}_ { r- })^\top \mathrm d W _ { r },  \\
Y^{y,Z}_{s}  =  y    -\int_{t}^{s} f (r, X^{t,x}_{r-}, Y^{y,Z}_{r-}, Z_{r} ) \mathrm d r+ \int_{t}^{s} (Z_{r})^\top \mathrm d W_{r}, \hspace{1cm}\hfill s\in[t,T],
\end{cases}
\end{equation}
and the control problem \eqref{eq:stochControlProblem} reduces to an optimization with respect to  $y\in \RR$ and $Z\in \mathbb{H}_{T}^{2}$.
Given $M\in\mathbb{N}$, a time discretization $0\leq t=t_0< t_1<\ldots <t_M=T$ is introduced. For simplicity let us assume that the mesh grid is uniform with step $\Delta t>0$. Let $\Delta W_{n}= W_{t_{n+1}} - W_{t_n}$ denote the increments of the Brownian motion. An Euler-Maruyama discretization of the FBSDE is considered, i.e. for $n =0,\ldots, M-1$ we define the following discrete time version of the dynamics \eqref{dyn_no_jumps},
\begin{equation}\label{EM_no_jumps}
\begin{cases}
{\widetilde X}_{{n+1}}    & =  {\widetilde X}_{n}  + b({\widetilde X}_{n}) \Delta t + \sigma({\widetilde X}_{n})  \Delta W_{n},\hfill \widetilde X_{0} = x,\\
{\widetilde Y}_{{n+1}} &=  {\widetilde Y}_{n} - f(t_n, {\widetilde X}_{n}, {\widetilde Y}_{n}, \widetilde{Z}_{n}) \Delta t + \widetilde{Z}_{n}^\top\Delta W_{n}, \hfill\hspace{1cm} {\widetilde Y}_0 = y. 
\end{cases}
\end{equation}
The next step is to represent, for each $n$, the control process $\widetilde Z_n$ in \eqref{EM_no_jumps} by using an artificial neural network (ANN). In particular,  feedforward ANN  with $\mathcal{L}+1\in\mathbb{N}\setminus \{1,2\}$ layers are employed. Each layer consists of $\upsilon_\ell$ \textit{nodes} (also called \textit{neurons}), for $\ell=0,\ldots,\mathcal{L}$. The $0$-th layer represents the \textit{input layer}, while the $\mathcal{L}$-th 
layer is called the \textit{output layer}. The remaining $\mathcal{L}-1$ layers are \textit{hidden layers}. For simplicity, we set $\upsilon_\ell=\upsilon$, $\ell=1,\ldots,\mathcal{L}-1$. In our setting, the dimension of the  input layer is set equal to $d$, i.e. the dimension of the forward process $X$.

A feedforward neural network is a function defined via the composition
\begin{align*}
\mathbb{R}^{d}\ni x  \longmapsto \mathcal{A}_{\mathcal{L}} \circ \varrho \circ \mathcal{A}_{\mathcal{L}-1} \circ \ldots \circ \varrho \circ \mathcal{A}_{1}(x),
\end{align*}  
where all $\mathcal{A}_\ell$, $\ell=1,\ldots,\mathcal{L}$, are affine transformations of the form
$\mathcal{A}_\ell(x):=\mathcal{W}_\ell x + \beta_\ell$,  $\ell=1,\ldots,\mathcal{L}$,
where $\mathcal{W}_\ell$ and $\beta_\ell$ are matrices and vectors of suitable size called, respectively,  weights and biases. The function $\varrho$, called \textit{activation function}, is a univariate function $\varrho\colon\mathbb{R}\to\mathbb{R}$ that is applied component-wise to vectors. With an abuse of notation, we denote
$\varrho(x_1,\ldots,x_\upsilon)=\left(\varrho(x_1),\ldots,\varrho(x_\upsilon)\right).$
The elements of $\mathcal{W}_\ell$ and $ \beta_\ell$ are the parameters of the neural network. One can regroup all parameters in a vector of size $R=\sum_{\ell=0}^\mathcal{L}\upsilon_\ell(1+\upsilon_\ell)$.

Let, for each $n=0,\ldots, M-1$,  $\mathcal Z^{{\rho_n}}_n:\RR^d\to \RR^d$ be  an ANN with parameters ${\rho_n}\in \RR^{R}$. Replacing $Z_n$ with $\mathcal Z^{{\rho_n}}_n({\widetilde X}_n)$ in \eqref{EM_no_jumps} one obtains the following  dynamics for the process $\widetilde Y$ parametrized by $\rho{=(\rho_0,\ldots, \rho_{M-1})\in (\RR^{R})^M}$
\begin{equation}\label{ANN_no_jumps}
\widetilde Y^{y,\rho}_{n+1} =  {\widetilde Y^{y,\rho}}_{n} - f(t_n, {\widetilde X}_{n}, \widetilde Y^{y,\rho}_{n}, \mathcal{Z}^{{\rho_n}}_{n}({\widetilde X}_{{n}})) \Delta t + (\mathcal{Z}^{{\rho_n}}_{n}({\widetilde X}_{{n}}))^\top\Delta W_{n}, \hfill\hspace{1cm} {\widetilde Y^{y,\rho}}_0 = y. 
\end{equation}
Observe that the solver involves the introduction of $M$ distinct neural networks, one network for each time step. 

The family of neural network is then trained simultaneously over a set of simulated Monte Carlo samples of the dynamics via a standard training algorithm such as stochastic gradient descent or its extensions (one notable example in this sense being given by the ADAM algorithm, see \cite{adam2015}) in order to minimize with respect to $y\in \RR$ and $\rho \in \RR^{R({\mathcal Z})} $ the loss function 
\begin{equation}\label{loss_nojumps}
\EE\left[\left(g(\widetilde X_{M})-\widetilde Y^{y,\rho}_{M}\right)^{2}\right].
\end{equation}

\section{Deep Solver with jumps}\label{sec:algorithm}
The aim of this section is to present our proposed extension of the Deep BSDE solver of \cite{ehanjen17} to cover the case of the FBSDE with jumps \eqref{eq:forward}-\eqref{eq:backward}.

The main challenge in view of the extension is represented by the case where the forward (and consequently the backward) process exhibits infinitely many jumps, be it of finite or infinite variation. Clearly, such jumps cannot be simulated exactly on a pre-specified grid of time instants. This problem has been first encountered in the literature on the discretization and simulation of L\'evy-driven stochastic differential equations, see \cite{fournier2011} among others. For the reader's convenience, we split the presentation of our solver in two subsections. In the first subsection we only consider the case with finitely many jumps since it can be simulated exactly via an Euler type discretization. In the second subsection we present the additional steps that we propose in order to cover the infinite activity case.

\subsection{Finite activity case}\label{sec:finite}
In this subsection we assume that the L\'evy measure satisfies the condition $\nu(\RR^d)<\infty$, meaning that the jumps are those of a compound Poisson process. We fix $M\in\NN$ and introduce a uniform time discretization $0\leq t=t_0<t_1<\ldots<t_M=T$ with time step $\Delta t =t_{n+1}-t_n$. As a starting point, we consider the forward SDE \eqref{eq:forward} between two subsequent time steps $t_n$ and $t_{n+1}$, namely,
\begin{align*}
X _ { t_{n+1} } &  = X _ { t_{n} } + \int _ { t_n } ^ { t_{n+1} } b \left( X _ { r- } \right) \mathrm d r + \int _ { t_n } ^ { t_{n+1} } \sigma \left( X _ { r- } \right)^\top \mathrm d W _ { r }  +  
\int^{t_{n+1}}_{t_n} \int_{\RR^d} \Gamma \left( X _ { r- } , z\right) \tilde N(\mathrm d r, \mathrm d z),
\end{align*}
 with $X _ { t_{0} } =   x \in \mathbb { R } ^ { d }$.
We freeze the coefficients between two consecutive time steps introducing the discrete time process $\widetilde X=(\widetilde X_n)_{n=0,\ldots, M}$, defined recursively by 
$$
\widetilde{X}_{n+1 }  = \widetilde{X} _ { {n} } + \int _ { t_n } ^ { t_{n+1} } b ( \widetilde{X} _ { n } ) \mathrm d r + \int _ { t_n } ^ { t_{n+1} } \sigma ( \widetilde{X} _ {n } )^\top \mathrm d W _ { r }  +  
\int^{t_{n+1}}_{t_{n}} \int_{\RR^d} \Gamma ( \widetilde{X} _ { n } , z) \tilde N(\mathrm d r, \mathrm d z),\qquad
n=0,\ldots, M-1,
$$
with $\widetilde{X} _ {0}  = x.$
We write  $N([0,t_n],\RR^d)$ to denote the number of $\RR^d$-valued jumps occurring over the time interval $[0,t_n]$. The previous equation can then  be rewritten as
\begin{align}
\label{eq:forwardEuler}
\begin{aligned}
\widetilde{X} _ {n+1} &  = \widetilde{X} _ {n} + b ( \widetilde{X} _{n} ) \Delta t + \sigma ( \widetilde{X} _ { n } )^\top \Delta W _ { n} \\
&\quad +  \sum_{i=N([0,t_n],\RR^d)+1}^{N([0,t_{n+1}],\RR^d)}\Gamma ( \widetilde{X} _ {n } , z_i)
-\Delta t  \int_{\RR^d} \Gamma ( \widetilde{X} _ {n } , z) \nu(\mathrm d z), \qquad
n=0,\ldots, M-1.
\end{aligned}
\end{align}

We proceed similarly for the backward dynamics  \eqref{eq:backward}. We consider again two consecutive time instants $t_n$ and $t_{n+1}$
\begin{align*}
 Y_{t_{n+1}}    = &   Y_{t_{n}} -\int_{t_{n}}^{t_{n+1}} f \Big(r, X_{r-}, Y_{r-}, Z_{r}, \int_{\RR^d} U_r(z) \nu(\mathrm d z) \Big) \mathrm d r+ \int_{t_{n}}^{t_{n+1}} Z_{r}^\top \mathrm d W_{r}  + \int_{t_{n}}^{t_{n+1}} \int_{\RR^d} U_r(z) \tilde N(\mathrm d r, \mathrm d z). 
\end{align*}
We now make use of the Feynman-Kac Theorem \ref{th:FeynmanKac}, meaning that above dynamics can be read as follows 
\begin{align}
\label{eq:BSDE_tn}
\begin{split}
 Y_{t_{n+1}}=&Y_{t_{n}} -\int_{t_{n}}^{t_{n+1}} f \Big(r, X_{r-}, Y_{r-}, \sigma({X}_{r-})^\top D_x u(r, {X}_{r-}) , \int_{\RR^d} u(r,X_{r-}+\Gamma(X_{r_-},z))-u(r,X_{r-}) \nu(\mathrm d z) \Big) \mathrm d r\\
 &+\int_{t_{n}}^{t_{n+1}} (D_x u(r, {X}_{r-}))^\top\sigma({X}_{r-}) \mathrm d W_{r} + \int_{t_{n}}^{t_{n+1}} \int_{\RR^d} u(r,X_{r-}+\Gamma(X_{r_-},z))-u(r,X_{r-}) \tilde N(\mathrm d r, \mathrm d z).
\end{split}
\end{align}
Next, we freeze again the coefficients and consider the following discrete time approximation

\begin{align}
\label{eq:disc_BSDE}
\begin{split}
 \widetilde{Y}_{{n+1}}    = &    \widetilde{Y}_{{n}} -f \Big(t_{n}, \widetilde{X}_{{n}}, \widetilde{Y}_{{n}},  \sigma(\widetilde {X}_{n})^\top D_x u(t_n, \widetilde {X}_{n}), \int_{\RR^d} u(t_n,\widetilde X_{n}+\Gamma(\widetilde X_{n},z))-u(t_n,\widetilde X_{n}) \nu(\mathrm d z)   \Big) \Delta t\\
 & + (D_x u(t_n, \widetilde {X}_{n}))^\top\sigma(\widetilde {X}_{n})  \Delta W_{n}  +  \Delta u(t_n,\widetilde{X}_n, \Delta \widetilde{X}_n)-  C(t_n,\widetilde{X}_n)\Delta t.\\
 \end{split}
\end{align}
where we have defined the jumps of the forward SDE, the jumps of the BSDE and compensator of the BSDE, respectively as
\begin{align}
\label{eq:SDE_jump}
\Delta \widetilde{X}_n &= \sum_{i=N([0,t_n],\RR^d)+1}^{N([0,t_{n+1}],\RR^d)}\Gamma(\widetilde{X}_{n},z_i),\\
\label{eq:BSDE_jump}
    \Delta u(t_n,\widetilde{X}_n, \Delta \widetilde{X}_n)&=u(t_n,\widetilde{X}_{n}+\Delta \widetilde{X}_n)-u(t_n,\widetilde{X}_{n}),\\ \label{eq:compensator}
    C(t_n,\widetilde{X}_n)&=\int_{\RR^d} u(t_n,\widetilde{X}_{n}+\Gamma(\widetilde{X}_{n},z))-u(t_n,\widetilde{X}_{n}) \nu(\mathrm d z).
\end{align}

At this point, we introduce the neural networks used to approximate the solution of the FBSDE \eqref{eq:forward}-\eqref{eq:backward}. Similar to the algorithm proposed in \cite{Han2018}, we introduce two single trainable parameters to approximate the deterministic initial value $Y_0$, and the deterministic gradient $D_x u(0,x)$, respectively. Moreover, we use a sequence of neural networks to approximate $D_x u$ at the discrete time points $\{t_1,t_2,\ldots,t_{M-1}\}$. Since we consider BSDEs with jumps we have to approximate additional terms in the BSDE (the last two terms in \eqref{eq:disc_BSDE}). Below we introduce all approximations and explain the algorithm. 

We employ first three sequences of neural networks to parametrize the functions $D_x u$, $\Delta u$ and $C$. Thus, for $n=1,2,\ldots,M-1$, we introduce the mappings $\mathcal{D}_x\mathcal{U}_{n}^{\rho_n}:\RR^d\to \RR^d$, $\Delta \mathcal{U}_n^{\varphi_n}\colon\RR^{2d}\to\RR$ and $ \mathcal{C}_n^{\vartheta_n}\colon\RR^{d}\to\RR$ parametrized by $\rho_n\in\RR^{R_\rho}$, $\varphi_n\in \RR^{R_\varphi}$ and $\vartheta_n\in\RR^{R_\vartheta}$, respectively. Here $R_\rho\in\NN$, $R_\varphi\in\NN$ and $R_\vartheta\in\NN$ represents the number of trainable parameters in each neural network. Note that we have not yet introduced approximations for $D_x u$, $\Delta u$ and $C$ at the initial time point $t_0$. The reason for this is that we consider a fixed, deterministic initial state $\widetilde{X}_0=x$ and hence $D_x u$ and $C$ are fixed, deterministic and $C$ depends only on potential jumps in $(t_0,t_1]$. Moreover, as stated above, in order to apply a forward Euler--Maruyama scheme on the BSDE we need an approximation of $Y_0$. Thus, we introduce the single trainable parameters $\widetilde{Y}_0\in\RR$, $\mathcal{D}_x\mathcal{U}_0\in\RR^d$ and $\mathcal{C}_0\in\RR$ and the neural network $\Delta \mathcal{U}_0^{\varphi_0}\colon \RR^d\to \RR$, parameterized by $\varphi_0\in\RR^{R_\varphi^0}$, where $R_\varphi^0\in\NN$ is the number of trainable parameters in $\Delta \mathcal{U}_0^{\varphi_0}$. To summarize, we have the trainable parameters $\widetilde{Y}_0\in\RR$, $\rho=\text{concat}(\mathcal{D}_x\mathcal{U}_0,\rho_1,\ldots,\rho_{M-1})\in \RR^{d+R_\rho\times(M-1)}$, $\varphi=\text{concat}(\varphi_0,\varphi_1,\ldots,\varphi_{M-1})\in\RR^{R_\varphi^0 + R_\varphi\times(M-1)}$ and $\vartheta=\text{concat}(\mathcal{C}_0,\vartheta_1,\ldots,\vartheta_{M-1})\in\RR^{d + R_\vartheta\times(M-1)}$, where $\text{concat}(\cdot)$ is a concatenation of a list of tensors into a vector. Finally, we define the collection of all trainable parameters as $\Theta=\text{concat}(\widetilde{Y}_0, \rho,\varphi,\vartheta)\in\RR^R$, where $R$ is the total number of trainable parameters in the full model. In Table \ref{tab:NN_summary}, we summarize all functions and constants that are approximated. 
\begin{table}[h]
\centering
\begin{tabular}{|l|l|l|}
\hline
\textbf{Function to Approximate}  & \textbf{Approximation Notation} & \textbf{Input} \\ \hline
\multicolumn{3}{|c|}{\textbf{Approximations at \( t = 0, \text{ i.e., } n=0\)}} \\ \hline
\( Y \)                    & \( \widetilde{Y}_0 \)      & None (constant) \\
\( D_x u \)            & \( \mathcal{D}_x \mathcal{U}_0 \) &  None (constant)  \\
\( \Delta u \)             & \( \Delta \mathcal{U}_0^{\varphi_0} \)  & \( \Delta \widetilde{X}_0 \) \\
\( C \)                    & \( \mathcal{C}_0 \)         & None (constant)  \\ \hline
\multicolumn{3}{|c|}{\textbf{Approximations for \( t > 0, \text{ i.e., } n\in\{1,2,\ldots,M-1\} \)}} \\ \hline
\( D_x u \) & \( \mathcal{D}_x \mathcal{U}_n^{\rho_n} \) & \( \widetilde{X}_n \) \\
\( \Delta u \) & \( \Delta \mathcal{U}_n^{\varphi_n} \) & \( \widetilde{X}_n,\  \Delta \widetilde{X}_n \) \\
\( C \) & \( \mathcal{C}_n^{\vartheta_n} \) & \( \widetilde{X}_n \) \\ \hline
\end{tabular}
\caption{{Summary of neural network approximations of the functions in the BSDE with jumps.}}
\label{tab:NN_summary}
\end{table}

The aim is then to construct a loss function such that when minimized with respect to $\Theta$, our model accurately approximates the solution of the BSDE. Our starting point is the mean squared error of the terminal condition, which is the entire loss function used in \cite{Han2018} for regular BSDEs. In the presence of jumps, we must also ensure that both the jumps in the BSDE and the compensator are accurately captured by our neural network approximations. By comparing \eqref{eq:disc_BSDE} with \eqref{eq:BSDE_tn} it is clear that we want to approximate \begin{equation*}
   \sum_{n=0}^{M-1} \int_{t_{n}}^{t_{n+1}} \int_{\RR^d} u(r,X_{r-}+\Gamma(X_{r_-},z))-u(r,X_{r-}) \tilde N(\mathrm d r, \mathrm d z)
\end{equation*} 
with
\begin{equation*}
        \Delta \mathcal{U}_0^{\varphi_0}(\Delta \widetilde{X}_0)-\mathcal{C}_0\Delta t+\sum_{n=1}^{M-1}\Delta\mathcal{U}_n^{\varphi_n}(\widetilde{X}_n,\Delta \widetilde{X}_n)-\mathcal{C}_n^{\vartheta_n}(\widetilde{X}_{n})\Delta t.
\end{equation*}
Recall that the stochastic integral with respect to the compensated Poisson random measure is a square integrable martingale under our assumptions, which implies that
\begin{align*}
 \mathbb{E}&\left[\int_{t_{n}}^{t_{n+1}} \int_{\RR^d} u\left(r,X_{r-}+\Gamma(X_{r_-},z)\right)-u\left(r,X_{r-}\right) \tilde{N}(\mathrm d r, \mathrm d z)\right]=0.
\end{align*}
To enforce this structure on our approximations we introduce an additional term in the loss function based on
\begin{equation*}
        \mathbbm{E}\Big[\Delta \mathcal{U}_0^{\varphi_0}(\Delta \widetilde{X}_0)-\mathcal{C}_0\Delta t\Big]^2+\sum_{n=1}^{M-1}\mathbbm{E}\Big[\Delta\mathcal{U}_n^{\varphi_n}(\widetilde{X}_{n},\Delta \widetilde{X}_n)-\mathcal{C}_n^{\vartheta_n}(\widetilde{X}_{n})\Delta t\Big]^2.
\end{equation*}

This concludes the description of the algorithm for the finite activity case. To summarize, our approach considers the discretized stochastic optimal control problem below. For some constant $K>0$
\begin{align}\label{eq:SOCP}
&\begin{aligned}
 \underset{\Theta\in\RR^R}{\text{minimize}}\quad &\mathbb{E}\left[\left(g(\widetilde{X}_{M})-\widetilde Y_{M}^\Theta\right)^{2}\right] + K\times\sum_{n=0}^{M-1}\mathbb{E}\big[\Delta \mathcal{U}_n^\Theta - \mathcal{C}_n^\Theta\Delta t\big]^2
\end{aligned}\\
&\text{subject to:}\nonumber\\
&
\begin{cases}
\widetilde{X}_0=&x,\quad \widetilde{Y}_0^\Theta = \widetilde{Y}_0,\quad \mathcal{Z}_0^\Theta=\sigma(x)^\top \mathcal{D}_x\mathcal{U}_0,\quad \Delta\mathcal{U}_0^\Theta=\Delta \mathcal{U}_0(\Delta \widetilde{X}_0),\quad \mathcal{C}_0^\Theta=\mathcal{C}_0
\vspace{0.05cm}\\
\widetilde{X}_{n+1}=&\widetilde{X} _ {n} + b ( \widetilde{X}_n) \Delta t + \sigma ( \widetilde{X} _ { n } )^\top \Delta W_n  +  \Delta \widetilde{X}_n -\Delta t  \int_{\RR^d} \Gamma ( \widetilde{X}_n , z) \nu(\mathrm d z)\nonumber\\
\widetilde {Y}^\Theta_{n+1}    = &   \widetilde{Y}^\Theta_n -f \left(t_{n}, \widetilde{X}_{n}, \widetilde{Y}^\Theta_n,  \mathcal{Z}_n^\Theta, \mathcal{C}_n^\Theta\right) \Delta t + \mathcal{Z}_n^\Theta \Delta W_{n} + \Delta\mathcal{U}_n^\Theta - \mathcal{C}_n^\Theta\Delta t \nonumber \\
\mathcal{Z}_n^\Theta=&\sigma(\widetilde{X}_n)^\top\mathcal{D}_x\mathcal{U}_n^{\rho_n}(\widetilde{X}_n),\quad \Delta\mathcal{U}_n^\Theta=\Delta \mathcal{U}_n^{\varphi_n}(\widetilde{X}_n,\Delta \widetilde{X}_n),\quad \mathcal{C}_n^\Theta=\mathcal{C}_n^{\vartheta_n}(\widetilde{X}_n).
\end{cases}
\end{align}
To avoid introducing overly cumbersome notation, we have allowed a slight abuse of notation in the above equation, as the last line is valid only for $n\geq 1$.

\begin{remark}
Note that $\Delta u$, which represents changes in $u$ due to jumps in $X$, is identically zero when no jumps occur between adjacent time points. Thus, it is reasonable to structure $\Delta \mathcal{U}_n^{\varphi_n}(\widetilde{X}_n, \Delta \widetilde{X}_n)$ to enforce this property. One way to achieve this is by defining $\Delta\mathcal{U}_n^\Theta = \Delta \mathcal{U}_n^{\varphi_n}(\widetilde{X}_n, \Delta \widetilde{X}_n) \Delta \widetilde{X}_n$. This approach is used for all numerical experiments in this paper, however, alternative methods could be considered, particularly for high-dimensional settings with complex terminal conditions. For example, a more general alternative might involve an indicator function that returns one if a jump occurs in any of the $d$ dimensions of $X$ and zero otherwise.
\end{remark}

\color{black}
\subsection{Infinite activity case}\label{sec:infinite}

To cover the case where the forward process $X$ exhibits infinitely many jumps, we proceed by introducing an approximating jump diffusion $X^\epsilon$ with finitely many jumps, meaning that the jump component corresponds to a compound Poisson process. The small jumps might be truncated or, as we do, they could be approximated via the diffusion term by increasing the volatility. This constitutes the first source of numerical error of our algorithm.  This error is empirically analyzed in the numerical results section\color{black}. Concerning this first approximation step for the forward SDE for $X$ we follow, among others \cite{ar2001}, \cite{cr2007} and \cite{jum2015}

Moving from the process $X$ to the process $X^\epsilon$ has an implication on the FBSDE as well. To proceed, we need to introduce some additional assumptions, in particular regarding the structure of the Poisson random measure: let $\epsilon\in(0,1]$. We define 
\begin{align*}
\nu^{\epsilon}(\de z):=\mathbbm{1}_{\{|z|>\epsilon\}} \nu(\de z),\quad \nu_{\epsilon}(\de z):=\mathbbm{1}_{\{|z|\leq\epsilon\}} \nu(\de z).
\end{align*}
It then holds that $\nu=\nu_{\epsilon}+\nu^{\epsilon}$.
From the definition of the L\'evy measure $\nu$ we immediately have $\int_{\mathbb{R}^{d}}|z|^{2} \nu_{\epsilon}(d z)< \infty \;$ and $\; \nu^{\epsilon}\left(\mathbb{R}^{d}\right)<\infty \;$.
We are factorizing the jump term in two components, the first one corresponding to the \textit{small} jumps, and the second one to the \textit{big} jumps, i.e. a compound Poisson process. The factorization of the L\'evy measure means that we are assuming that we can write the Poisson Random Measure $N$  as 
\begin{align}
N=N_\epsilon+N^{\epsilon} \; ,
\end{align}
with $N_\epsilon$ and $N^{\epsilon}$ having the compensator $\nu_{\epsilon}$ and $\nu^{\epsilon}$, respectively.

In line with \cite{cr2007} we introduce the following quantity
\begin{align*}
\Sigma_\epsilon:=\int_{\RR^d}zz^\top \nu_\epsilon(\de z)=\int_{|z|<\epsilon}zz^\top \nu(\de z),
\end{align*} 
taking values in the cone of positive semidefinite $d\times d$ matrices $S^d_+$. We also introduce a further assumption on the jump term of the forward SDE:
\begin{itemize}
\item[(A5)] 
the function  $\Gamma\colon\RR^d\times \RR^d\to \RR^d$ is of the form
$$
\Gamma(x,z) \coloneq \gamma(x) z
$$
for some Lipschitz continuous function $\gamma:\RR^d\to\RR^{d\times d}$.
\end{itemize}
The remaining assumptions on the vector fields $b,\sigma,\Gamma$ are unchanged with respect to \eqref{eq:forward}.

Given $(t,x)\in [0,T]\times \RR^d$, we approximate the solution $X^{t,x}_\cdot$ of the forward SDE \eqref{eq:forward} by means of the process $X^{\epsilon,t,x}_\cdot$, whose dynamics are
\begin{align}
\begin{aligned}
X^{\epsilon,t,x} _ { s }   = &  x + \int _ { t } ^ { s } b( X^{\epsilon,t,x}  _ { r- } ) \mathrm d r + \int _ { t } ^ { s }\Big( \sigma ( X^{\epsilon,t,x}  _ { r- } )^\top
 +\gamma ( X^{\epsilon,t,x}  _ { r- } )\sqrt{\Sigma_\epsilon}\Big) \mathrm d W _ { r }+  
\int^s_t \int_{\RR^d} \gamma ( X^{\epsilon,t,x}  _ { r- } )z \tilde N^{\epsilon}(\mathrm d r, \mathrm d z),
 \label{eq:forwardEps}
 \end{aligned}
\end{align}
where $\sqrt{\Sigma_\epsilon}$ denotes the matrix square root of $\Sigma_\epsilon\in S^d_+$ and we set $\tilde N^{\epsilon}\coloneq N^{\epsilon}(\mathrm d r, \mathrm d z)-\nu^\epsilon(\mathrm d z)\mathrm d r$. Again, whenever possible, we will simply use the notation $X^\epsilon$, thus neglecting the dependence on the starting point  $(t,x)$. 
A priori error estimates of the error we introduce by approximating $X$ via $X^{\epsilon}$ can be derived following, for instance,  \cite{cr2007} and \cite{jum2015} and will be discussed in the next section.

At this point we apply the solver as described in Section \ref{sec:finite} to the following BSDE
\begin{align}
 Y^{\eps\eps}_{t}    = &  g (X^\epsilon_T) +\int_{t}^{T} f (s, X^\epsilon_{r-}, Y^{\eps\eps}_{r-}, Z^{\eps\eps}_{r}, \int_{\RR^d} U^{\eps\eps}(r,z) \nu^\epsilon(\mathrm d z) ) \mathrm d r- \int_{t}^{T} {Z^{\eps\eps}_{r}}^{\top}  \mathrm d W_{r} \nonumber \\
& - \int^T_t \int_{\RR^d} U^{\eps\eps}_r(z) \tilde N^{\epsilon}(\mathrm d r, \mathrm d z). \label{eq:backwardepseps}
\end{align}
whose solution $(Y^{\eps\eps},Z^{\eps\eps}, U^{\eps\eps})$ is intended to approximate $(Y,Z,U)$ solving \eqref{eq:backward}. {By the above approximation, we are left with a finite activity FBSDE, which can be approximated with the algorithm outlined in Section \ref{sec:finite}.}

\begin{remark}
    The  use of a double superscript ${\ }^{\eps\eps}$ in \eqref{eq:backwardepseps} is motivated by the introduction, later in  Section \ref{sec:errorBack}, of an intermediate approximation step denoted with a single superscript ${\ }^\eps$.
\end{remark}

\section{Error estimates}\label{sec:errorEstimates}
In this section we present a priori error estimates for the error introduced by the approximation of small jumps in the infinite activity case on the solution of the FBSDE. 
\subsection{A priori error estimates for the forward approximation}\label{sec:errorforward}
We start by providing the $L^2$ estimate of the error that we introduce by approximating $X_\cdot$ via $X^{\epsilon}_\cdot$.
 We proceed along the lines of  \cite{cr2007} and \cite{jum2015}. {We report the main step of the proof  for the sake of completeness.  }

\begin{theorem}[{Theorem 4.4, \cite{jum2015}}]\label{teo:est_forward}
Under the assumptions (A1)-(A2), (A5), there exists a constant $C>0$ such that
\begin{align}
\mathbb{E}\left[\sup_{s\in [t,T]}\left|X_s-X^\epsilon_s \right|^2\right]\leq C(1+|x|^2)\int_{\RR^d}|z|^2\nu_\epsilon(\de z),\qquad \forall (t,x)\in [0,T]\times \RR^d.
\end{align}
\end{theorem}

\begin{proof} {In the following we do not keep track of constants. 
Applying Theorem 2.11 in \cite{Kunita2004} with $p=2$, we get the following estimate
\begin{align*}
\mathbb{E}&\left[\sup_{s\in [t,T]}\left|X_s-X^\epsilon_s \right|^2\right]\\
&\leq C\bigg\{\mathbb{E}\left[\int_t^T \left|b \left( X _ { r- } \right)-b \left( X^\epsilon _ { r- } \right)\right|^2\de r\right]+ \mathbb{E}\left[\int_t^T \left|\sigma \left( X _ { r- } \right)^\top-\sigma \left( X^\epsilon _ { r- } \right)^\top-\gamma \left( X^\epsilon _ { r- } \right)\sqrt{\Sigma_\epsilon} \right|^2\de r\right]\\
&+\mathbb{E}\left[\int_t^T  \int_{\RR^d} \left|\gamma \left( X _ { r- } \right)-\gamma \left( X^\epsilon _ { r- } \right)\right|^2\left|z\right|^2\nu(\de z)\de r\right]+\mathbb{E}\left[\int_t^T  \int_{\RR^d} \left|\gamma \left( X^\epsilon _ { r- } \right)\right|^2\left|z\right|^2\nu_\epsilon(\de z)\de r\right]\bigg\}.
\end{align*}
Exploiting the Lipschitz property of the coefficient functions $b,\sigma,\gamma$, one obtains
\begin{align*}
\mathbb{E}&\left[\sup_{s\in [t, T]}\left|X_s-X^\epsilon_s \right|^2\right]\\
& {\leq  C\left\{\left(1 + \int_{\mathbb{R}^{d}}|z|^{2} \nu(d z)\right)\mathbb{E}\left[\int_t^T \left|X _ { r }-X^\epsilon _ { r}   \right|\de r\right]  + \left(\left|\sqrt{\Sigma_\eps}\right|^2  + \int_{\mathbb{R}^{d}}|z|^{2} \nu_{\epsilon}(d z)\right)\mathbb{E}\left[ \int^T_0 1 + \left|X^\eps_r\right|^2 \de r \right] \right\} }
\end{align*}
{so that, observing that $\left|\sqrt{\Sigma_\eps}\right|^2= \int_{\mathbb{R}^{d}}|z|^{2} \nu_{\epsilon}(d z)$ (see \cite[Proposition 4.1.2]{jum2015}) and making use of the classical moment estimate 
\begin{align}\label{eq:growth}
\mathbb{E}\left[\sup_{ s\in [t,T]}\left|X^{\epsilon}_s \right|^2\right]\leq C\left(1+|x|^{2}\right),\qquad \forall (t, x)\in [0,T]\times \RR^d, \text{for some $C> 0$,}
\end{align}
we can show the existence of a  constant $C>0$ such that }
$$
\mathbb{E}\left[\sup_{s\in [t, T]}\left|X_s-X^\epsilon_s \right|^2\right] \leq  C\left\{\mathbb{E}\left[\int_t^T \left|X _ { r }-X^\epsilon _ { r}   \right|\de r\right]  + (1+|x|^2) \int_{\mathbb{R}^{d}}|z|^{2} \nu_{\epsilon}(d z) \right\}.
$$
The result follows by Gronwall's inequality.}
\end{proof}

\subsection{A priori error estimates for the backward approximation}\label{sec:errorBack}

 In this section we  estimate the error induced by the approximation of small jumps on the solution $(Y,Z,U)\in S^2_T\times \mathbb H_{T} \times \mathbb H_{T,N}$ to \eqref{eq:backward}.
Let $(Y^{\eps\eps},Z^{\eps\eps},U^{\eps\eps})\in S^2_T\times \mathbb H_{T} \times \mathbb H_{T,N^\eps}$ be the solution to \eqref{eq:backwardepseps}.
{The error between $(Y,Z,U)$ and $
(Y^{\eps\eps},Z^{\eps\eps},U^{\eps\eps})$ can be seen as the result of two contributions: the error coming from the forward approximation (i.e. the replacement of $X$ with $X^\eps$) and the one related to the replacement of the infinite activity driving jumps of the backward dynamics \eqref{eq:backward} with the finite activity one appearing in  \eqref{eq:backwardepseps}. To estimate the first contribution to the error one can combine the results of Section \ref{sec:errorforward} with classical continuous dependence results for FBSDEs as those in \cite{Delongbook}. The main part of the proof is devoted to estimate  the second error term. We point out that similar estimates for the first component of the BSDE ($Y$) have previously  obtained in \cite{drz2021} by means of probabilistic techniques.}

We have the following result 
\begin{theorem}\label{teo:errorback}
Under assumptions (A1)-(A5), there exists a constant $C>0$ such that the following estimates hold
$$
\left\|Y-Y^{\eps\eps}\right\|_{\mathbb{S}^{2}}^{2} + \left\|Z-Z^{\eps\eps}\right\|_{\mathbb{H}^{2}}^{2} +\left\|U-U^{\eps\eps}\right\|_{\mathbb{H}^{2}_{N^\eps}}^{2} \leq C (1+|x|^2) \int_{\RR^d} |z|^2 \nu_\eps (\de z)\qquad \forall x\in \RR^d.
$$
for some positive constant $C$.
\end{theorem}

\begin{proof}
Let us consider, in addition to \eqref{eq:backward} and \eqref{eq:backwardepseps}, a third BSDE where we only substitute the forward process with the jump diffusion approximation and keep the original Poisson measure $N$ for the jump component:
\begin{equation}\label{eq:backwardeps}
\begin{aligned}
 Y^\epsilon_{t}    = &  g (X^\epsilon_T) +\int_{t}^{T} f \left(s, X^\epsilon_{s-}, Y^\epsilon_{s-}, Z^\epsilon_{s}, \int_{\RR^d} U^\epsilon_s(z) \nu(\mathrm d z) \right) \mathrm d s- \int_{t}^{T} {Z_{s}^{\epsilon}}^\top \mathrm d W_{s}  \\
& - \int^T_t \int_{\RR^d} U^\epsilon_s(z) \tilde N(\mathrm d s, \mathrm d z). 
\end{aligned}
\end{equation}

We clearly have 
\begin{align}\label{eq:Y-Yee}
\begin{aligned}
\left\|Y-Y^{\eps\eps}\right\|_{\mathbb{S}^{2}}^{2} \coloneq 
& \mathbb{E}\left[\sup_{ t \in [0,T]}\left|Y_{t} -Y^{\epsilon\epsilon}_{t}\right|^2\right] =\mathbb{E}\left[\sup_{t \in [0,T]}\left|Y_{t}-Y^{\epsilon}_{t}+Y^{\epsilon}_{t} -Y^{\epsilon\epsilon}_{t}\right|^2\right]
\\
&
\leq2 \left(\left\|Y-Y^{\eps}\right\|_{\mathbb{S}^{2}}^{2}+\left\|Y^\eps-Y^{\eps\eps}\right\|_{\mathbb{S}^{2}}^{2}\right)
\end{aligned}
\end{align}
\begin{align}\label{eq:Z-Zee}
\begin{aligned}
\left\|Z-Z^{\eps\eps}\right\|_{\mathbb{H}^{2}}^{2} \coloneq 
& \mathbb{E}\left[\int^T_0 \left|Z_{t} -Z^{\epsilon\epsilon}_{t}\right|^2 \de t \right] =\mathbb{E}\left[\int^T_0 \left|Z_{t}-Z^{\epsilon}_{t}+Z^{\epsilon}_{t} -Z^{\epsilon\epsilon}_{t}\right|^2\de t\right]\\
&
\leq2 \left( \left\|Z-Z^{\eps}\right\|_{\mathbb{H}^{2}}^{2} + \left\|Z^\eps-Z^{\eps\eps}\right\|_{\mathbb{H}^{2}}^{2}\right)
\end{aligned}
\end{align}
\begin{align}\label{eq:U-Uee}
\begin{aligned}
\left\|U-U^{\eps\eps}\right\|_{\mathbb{H}^{2}_{N^\eps}}^{2} \coloneq 
& \mathbb{E}\left[\int^T_0 \int_{\RR^d} \left|U_t (z)-U^{\epsilon\epsilon}_t(z) \right|^2 \nu^\eps(\de z) \de t \right] \\
& = \mathbb{E}\left[\int^T_0 \int_{\RR^d} \left|U_t(z) - U^\eps_t(z) + U^\eps_t(z) - U^{\epsilon\epsilon}_t(z)\right|^2 \nu^\eps(\de z) \de t \right]\\
& \leq 2 \left(\mathbb{E}\left[\int^T_0 \int_{\RR^d} \left|U_t (z)- U^\eps_t(z)\right|^2 \nu(\de z) \de t \right] +  \mathbb{E}\left[\int^T_0 \int_{\RR^{d}} \left|U^\eps_t(z) - U^{\eps\eps}_t(z)\right|^2 \nu^\eps(\de z) \de t \right]\right)\\
& \leq 2\left(\left\|U-U^{\eps}\right\|_{\mathbb{H}^{2}_N}^{2}  +\left\|U^\eps-U^{\eps\eps}\right\|_{\mathbb{H}^{2}_{N^\eps}}^{2}  \right)
\end{aligned}
\end{align}
In what follows we will not keep track of the constants and we will denote them by a generic $C$.
We start by the first term in the right hand side of \eqref{eq:Y-Yee},  \eqref{eq:Z-Zee} and \eqref{eq:U-Uee}. 
For any $(\omega,s, y, z, u)\in \Omega\times [0,T]\times \RR\times \RR^d\times L^2_\nu(\RR^d)$, let us define
$$
F(\omega,s, y, z, u) := f(s, X_{s-}, y, z, \int_{\RR^d} u(z) \nu(\mathrm d z)) 
$$
and
$$
F^\eps(\omega,s, y, z, u) := f(s, X^\epsilon_{s-}, y, z, \int_{\RR^d} u(z) \nu(\mathrm d z)).
$$
Applying the estimates (3.5) and (3.7) in \cite{Delongbook} one has that for any $\rho>0$ there exists some constant $C$ such that 
 \begin{align}
\begin{aligned}
&\EE\left[\sup_{t\in [0,T]} e^{\rho t} |Y_t-Y^\eps_t|^2\right] + \EE\left[\int^T_0 e^{\rho t} |Z_t-Z^\eps_t|^2 \de t\right] +  \EE\left[\int^T_0 \int_{\RR^d} e^{\rho t} |U_t(z)-U^\eps_t(z)|^2\nu(\de z)\de t \right]\\
& \leq
C \left(\mathbb{E}\left[e^{\rho T}\left|g(X_T)-g(X^{\epsilon}_T)\right|^{2}\right]\right.\\
&\quad\left.+\mathbb{E}\left[\int_0^Te^{\rho t}\left|f(t,X_{t-}, Y_{t-}, Z_{t}, \int_{\RR^d} U_t(z) \nu(\mathrm d z))-f(t,X^\epsilon_{t-}, Y_{t-}, Z_{t}, \int_{\RR^d} U_t(z) \nu(\mathrm d z))\right|^2\de t\right]\right).
\end{aligned}
\end{align}
Thanks to the Lipschitz continuity of $g$ and $f$, see (A1) and (A2), and Theorem \ref{teo:est_forward} we then obtain 
 \begin{align*}
& \left\|Y-Y^{\epsilon}\right\|_{\mathbb{S}^{2}}^{2} + \|Z- Z^\eps\|_{\mathbb H^2} +  \|U- U^\eps\|_{\mathbb H^2_N}\\
& \leq \EE\left[\sup_{t\in [0,T]} e^{\rho t}  |Y_t-Y^\eps_t|^2\right| + \EE\left[\int^T_0 e^{\rho t} |Z_t-Z^\eps_t|^2 \de t\right] +  \EE\left[\int^T_0 \int_{\RR^d} e^{\rho t} |U_t(z)-U^\eps_t(z)|^2\nu(\de z)\de t \right]\\
& \leq C \EE\left[ \sup_{t\in [0,T]} \left|X_t - X^\eps_t\right|^2\right]\leq C(1+|x|^2)\int_{\RR^d}|z|^2\nu_\epsilon(\de z).
\end{align*}
Let us now move to the estimates for the terms $\left\|Y^{\eps}-Y^{\eps\eps}\right\|_{\mathbb{S}^{2}}^{2}$, $\left\|Z^{\eps}-Z^{\eps\eps}\right\|_{\mathbb{H}^{2}}^{2}$ and  $\left\|U^{\eps}-U^{\eps\eps}\right\|_{\mathbb{H}^{2}_{N^\eps}}^{2}$ in \eqref{eq:Y-Yee},  \eqref{eq:Z-Zee},  \eqref{eq:U-Uee}, respectively. First of all, let us observe that defining
 $$
\hat U^{\eps\eps}_s(z) := U^{\eps\eps}_s(z) \mathbbm 1_{|z|\geq \eps}
$$
we can write $Y^{\eps\eps}$ as 
\begin{align}
 Y^{\epsilon\epsilon}_{t}    = &  g (X^\epsilon_T) +\int_{t}^{T} f \left(s, X^\epsilon_{s-}, Y^{\epsilon\epsilon}_{s-}, Z^{\epsilon\epsilon}_{s}, \int_{\RR^d} \hat U^{\epsilon\epsilon}_s(z) \nu(\mathrm d z) \right) \mathrm d s- \int_{t}^{T} {Z_{s}^{\epsilon\epsilon}}^\top  \mathrm d W_{s} \nonumber \\
& - \int^T_t \int_{\RR^d} \hat U^{\epsilon\epsilon}_s(z) \tilde N(\mathrm d s, \mathrm d z). \label{eq:backwardepseps2}
\end{align}

We proceed in line with  \cite[Lemma 3.1.1]{Delongbook}.
We apply the Itô's formula to $e^{\rho t}\left|Y^\eps_{t}-Y^{\eps\eps}_{t}\right|^{2}$ and we derive, for any $\tau\in [t,T]$,
\begin{align}
\begin{aligned} 
\label{eq:ItoFormula}
&e^{\rho t}\left|Y^\eps_{t}-Y^{\eps\eps}_{t}\right|^{2}+\rho \int_{t}^{\tau} e^{\rho s}\left|Y^\eps_{s}-Y^{\eps\eps}_{s}\right|^{2}  \mathrm d s+\int_{t}^{\tau} e^{\rho s}\left|Z^\eps_{s}-Z^{\eps\eps}_{s}\right|^{2}  \de s\\
&+\int_{t}^{\tau} \int_{\mathbb{R}^d} e^{\rho s}\left|U^\eps_s( z)-\hat U^{\eps\eps}_s(z)\right|^{2} \nu(\de z)  \de s \\
=& e^{\rho \tau}\left|Y^\eps_{\tau}-Y^{\eps\eps}_{\tau}\right|^{2} \\ 
&-2 \int_{t}^{\tau} e^{\rho s}\left(Y^\eps_{s}-Y^{\epsilon\eps}_{s}\right) \\ 
& \cdot\left(-f(s,X^\eps_{s-}, Y^\eps_{s-}, Z^\eps_{s}, \int_{\RR^d} U^\eps_s(z) \nu(\mathrm d z))+f(s,X^\epsilon_{s-} Y^{\epsilon\eps}_{s-}, Z^{\epsilon\eps}_{s}, \int_{\RR^d} \hat U^{\epsilon\eps}_s(z) \nu(\mathrm d z))\right) \de s \\ 
&-2 \int_{t}^{\tau} e^{\rho s}\left(Y^\eps_{s-}-Y^{\epsilon\eps}_{s-}\right)\left(Z^\eps_{s}-Z^{\epsilon\eps}_{s}\right) \de W_{s} \\ 
&-2 \int_{t}^{\tau} \int_{\mathbb{R}^d} e^{\rho s}\left(Y^\eps_{s-}-Y^{\epsilon\eps}_{s-}\right)\left(U^\eps(s, z)-\hat U^{\epsilon\eps}(s, z)\right) \tilde{N}(\de s, \de z) \\ 
&-\int_{t}^{\tau} \int_{\mathbb{R}^d} e^{\rho s}\left|U^\eps_s( z)-\hat U^{\epsilon\eps}_s( z)\right|^{2} \tilde{N}(\de s, \de z), \quad 0 \leq t \leq \tau\leq T . \end{aligned}
\end{align}
The stochastic integrals appearing above are local martingales. Even more, they are uniformly integrable martingales. Taking the expectation of the inequality above and choosing $\tau=T$ we find 
\begin{align*}
\begin{aligned} 
&e^{\rho t}\EE\left[\left|Y^\eps_{t}-Y^{\eps\eps}_{t}\right|^{2}\right]+\rho \EE\left[ \int_{t}^{T} e^{\rho s}\left|Y^\eps_{s}-Y^{\eps\eps}_{s}\right|^{2}  \mathrm d s\right]+\EE\left[\int_{t}^{T} e^{\rho s}\left|Z^\eps_{s}-Z^{\eps\eps}_{s}\right|^{2}  \de s\right]\\
&+\EE\left[\int_{t}^{T} \int_{\mathbb{R}^d} e^{\rho s}\left|U^\eps_s(z)-\hat U^{\eps\eps}_s(z)\right|^{2} \nu(\de z)  \de s \right]\\ 
= &2 \EE\bigg[\int_{t}^{T} e^{\rho s}\left(Y^\eps_{s}-Y^{\epsilon\eps}_{s}\right) \\
& \cdot\left( f(s,X^\eps_{s-}, Y^\eps_{s-}, Z^\eps_{s}, \int_{\RR^d} U^\eps_s(z) \nu(\mathrm d z)) - f(s,X^\epsilon_{s-} Y^{\epsilon\eps}_{s-}, Z^{\epsilon\eps}_{s}, \int_{\RR^d} \hat U^{\epsilon\eps}_s(z) \nu(\mathrm d z))\right) \de s\bigg], \quad 0 \leq t \leq T . \end{aligned}
\end{align*}
We can use the  following inequality: for any $\alpha>0$.
\begin{align}
\label{eq:productInequality}
2|u v| \leq \frac{1}{\alpha}|u|^{2}+\alpha|v|^{2} .
\end{align}
We apply this to the term involving the difference between the drivers and obtain
\begin{align*}
& 2\left|Y^\eps_{s}-Y^{\epsilon\eps}_{s}\right|\left|f(s,X^\eps_{s-}, Y^\eps_{s-}, Z^\eps_{s}, \int_{\RR^d} U^\eps_s(z) \nu(\mathrm d z))-f(s,X^\epsilon_{s-} Y^{\epsilon\eps}_{s-}, Z^{\epsilon\eps}_{s}, \int_{\RR^d} \hat U^{\epsilon\eps}_s(z) \nu(\mathrm d z))\right|\\
& \leq\alpha\left|Y^\eps_{s}-Y^{\epsilon\eps}_{s}\right|^2+\frac{1}{\alpha}\left|f(s,X^\eps_{s-}, Y^\eps_{s-}, Z^\eps_{s}, \int_{\RR^d} U^\eps_s(z) \nu(\mathrm d z))-f(s,X^\epsilon_{s-} Y^{\epsilon\eps}_{s-}, Z^{\epsilon\eps}_{s}, \int_{\RR^d} \hat U^{\epsilon\eps}_s(z) \nu(\mathrm d z))\right|^2,
\end{align*}
so that  we get 
\begin{align}\label{eq:stima1}
\begin{aligned} 
&e^{\rho t}\EE\left[\left|Y^\eps_{t}-Y^{\eps\eps}_{t}\right|^{2}\right]+\rho \EE\left[ \int_{t}^{T} e^{\rho s}\left|Y^\eps_{s}-Y^{\eps\eps}_{s}\right|^{2}  \mathrm d s\right]+\EE\left[\int_{t}^{T} e^{\rho s}\left|Z^\eps_{s}-Z^{\eps\eps}_{s}\right|^{2}  \de s\right]\\
&\quad +\EE\left[\int_{t}^{T} \int_{\mathbb{R}^d} e^{\rho s}\left|U^\eps_s(z)-\hat U^{\eps\eps}_s(z)\right|^{2} \nu(\de z)  \de s \right]\\ 
& \leq \alpha \EE\left[ \int_{t}^{T} e^{\rho s}\left|Y^\eps_{s}-Y^{\eps\eps}_{s}\right|^{2}  \mathrm d s\right]  \\
& \quad + \frac{1}{\alpha}  \EE\bigg[\int_{t}^{T} e^{\rho s} \left|f(s,X^\epsilon_{s-}, Y^{\epsilon\eps}_{s-}, Z^{\epsilon\eps}_{s}, \int_{\RR^d} \hat U^{\epsilon\eps}_s(z) \nu(\mathrm d z)) -f(s,X^\eps_{s-}, Y^\eps_{s-}, Z^\eps_{s}, \int_{\RR^d} U^\eps_s(z) \nu(\mathrm d z))\right|^2 \de s\bigg]. \end{aligned}
\end{align}
Using the Lipschitz continuity of $f$ one has 
\begin{align*}
& \left|f(s,X^\epsilon_{s-}, Y^{\epsilon}_{s-}, Z^{\epsilon}_{s}, \int_{\RR^d}  U^{\epsilon}_s(z) \nu(\mathrm d z)) -f(s,X^\eps_{s-}, Y^{\eps\eps}_{s-}, Z^{\eps\eps}_{s}, \int_{\RR^d} \hat U^{\eps\eps}_s(z) \nu(\mathrm d z))\right|^2\\
& \leq C \left\{ \left|Y^\eps_{s-}-Y^{\eps\eps}_{s-}\right|^{2} + \left|Z^\eps_{s}-Z^{\eps\eps}_{s}\right|^{2} + \left|\int_{\RR^d} U^{\epsilon}_s(z)-\hat U^{\eps\eps}_s(z)\nu(\mathrm d z)\right|^{2}\right\}.
\end{align*}
Therefore, from \eqref{eq:stima1} one gets
\begin{align}\label{eq:stima2}
\begin{aligned} 
&e^{\rho t}\EE\left[\left|Y^\eps_{t}-Y^{\eps\eps}_{t}\right|^{2}\right]+\rho \EE\left[ \int_{t}^{T} e^{\rho s}\left|Y^\eps_{s}-Y^{\eps\eps}_{s}\right|^{2}  \mathrm d s\right]+\EE\left[\int_{t}^{T} e^{\rho s}\left|Z^\eps_{s}-Z^{\eps\eps}_{s}\right|^{2}  \de s\right]\\
&+\EE\left[\int_{t}^{T} \int_{\mathbb{R}^d} e^{\rho s}\left|U^\eps_s(z)-\hat U^{\eps\eps}_s( z)\right|^{2} \nu(\de z)  \de s \right]\\ 
 & \leq \left(\alpha + \frac{C}{\alpha}\right) \EE\left[ \int_{t}^{T} e^{\rho s}\left|Y^\eps_{s}-Y^{\eps\eps}_{s}\right|^{2}  \mathrm d s\right] +  \frac{C}{\alpha}  \EE\bigg[\int_{t}^{T} e^{\rho s} \left|Z^\eps_{s}-Z^{\eps\eps}_{s}\right|^{2} \de s\bigg] \\
& \quad + \frac{C}{\alpha}\EE\bigg[\int_{t}^{T} e^{\rho s} \left|\int_{\RR^d} (U^{\epsilon}_s(z)-\hat U^{\eps\eps}_s(z))\nu(\mathrm d z)\right|^{2} \de s\bigg]\\
& \leq \left(\alpha + \frac{C}{\alpha}\right) \EE\left[ \int_{t}^{T} e^{\rho s}\left|Y^\eps_{s}-Y^{\eps\eps}_{s}\right|^{2}  \mathrm d s\right] +  \frac{C}{\alpha}  \EE\bigg[\int_{t}^{T} e^{\rho s} \left|Z^\eps_{s}-Z^{\eps\eps}_{s}\right|^{2} \de s\bigg] \\
& \quad + \frac{2C}{\alpha}\EE\bigg[ \int_{t}^{T} e^{\rho s} \int_{\RR^d} \left|U^{\epsilon}_s(z)-U^{\eps\eps}_s(z)\right|^2 \nu^\eps(\mathrm d z) \de s\bigg] +\frac{2C}{\alpha}\EE\bigg[ \int_{t}^{T} e^{\rho s} \int_{\RR^d} \left|U^{\epsilon}_s(z)\right|^2 \nu_\eps(\mathrm d z) \de s \bigg]\\
\end{aligned}
\end{align}
Observing that 
$$
\EE\left[\int_{t}^{T} \int_{\mathbb{R}} e^{\rho s}\left|U^\eps_s(z)-\hat U^{\eps\eps}_s(z)\right|^{2} \nu(\de z)  \de s \right]\geq \EE\left[\int_{t}^{T} \int_{\mathbb{R}^d} e^{\rho s}\left|U^\eps_s(z)- U^{\eps\eps}_s( z)\right|^{2} \nu^\eps(\de z)  \de s \right]
$$
one obtains
\begin{align}\label{eq:stima3}
\begin{aligned} 
&e^{\rho t}\EE\left[\left|Y^\eps_{t}-Y^{\eps\eps}_{t}\right|^{2}\right]+\rho \EE\left[ \int_{t}^{T} e^{\rho s}\left|Y^\eps_{s}-Y^{\eps\eps}_{s}\right|^{2}  \mathrm d s\right]+\EE\left[\int_{t}^{T} e^{\rho s}\left|Z^\eps_{s}-Z^{\eps\eps}_{s}\right|^{2}  \de s\right]\\
&+\EE\left[\int_{t}^{T} \int_{\mathbb{R}^d} e^{\rho s}\left|U^\eps_s( z)- U^{\eps\eps}_s(z)\right|^{2} \nu^\eps(\de z)  \de s \right]\\ 
&  \leq \left(\alpha + \frac{C}{\alpha}\right) \EE\left[ \int_{t}^{T} e^{\rho s}\left|Y^\eps_{s}-Y^{\eps\eps}_{s}\right|^{2}  \mathrm d s\right] +  \frac{C}{\alpha}  \EE\bigg[\int_{t}^{T} e^{\rho s} \left|Z^\eps_{s}-Z^{\eps\eps}_{s}\right|^{2} \de s\bigg] \\
& \quad + \frac{C}{\alpha}\EE\bigg[ \int_{t}^{T} e^{\rho s} \int_{\RR^d} \left|(U^{\epsilon}_s(z)-U^{\eps\eps}_s(z))\right|^2 \nu^\eps(\mathrm d z) \de s\bigg] +\frac{C}{\alpha}\EE\bigg[ \int_{t}^{T} e^{\rho s} \int_{\RR^d} \left|U^{\epsilon}_s(z)\right|^2 \nu_\eps(\mathrm d z) \de s \bigg]
\end{aligned}
\end{align}
for a generic constant $C$ that majorates all constants appearing in \eqref{eq:stima2}. 
Assuming $\rho$ sufficiently big we can take  $\alpha$ such that $\rho = \alpha + \frac{C}{\alpha}$ and $C<\alpha$ in order to obtain  
$$
\begin{aligned} 
&\EE\left[\int_{t}^{T} e^{\rho s}\left|Z^\eps_{s}-Z^{\eps\eps}_{s}\right|^{2}  \de s\right] +\EE\left[\int_{t}^{T} \int_{\mathbb{R}^d} e^{\rho s}\left|U^\eps_s(z)- U^{\eps\eps}_s(z)\right|^{2} \nu^\eps(\de z)  \de s \right]\\ 
&  \leq C \EE\bigg[ \int_{t}^{T} e^{\rho s} \int_{\RR^d} \left|U^{\epsilon}_s(z)\right|^2 \nu_\eps(\mathrm d z) \de s \bigg].
\end{aligned}
$$
To estimate the right end side of the last inequality we can use the representation of $U^\eps$ given by Theorem \ref{th:FeynmanKac}. Indeed, for the starting  FBSDE \eqref{eq:backward} where we replace the forward component $X$ with $X^{\eps}$, we have   
$$
U^{\eps}(s, z)= u\left(s, {X}^{\eps}_{s-}+\Gamma\left({X}^{\eps}_{s-}, z\right)\right) -u\left(s, {X}^{\eps}_{s-}\right), \quad t \leq s \leq T, z \in \mathbb{R}^d
$$
so that by the Lipschitz continuity of $u$ (see \cite[Lemma 4.1.1]{Delongbook}) and the growth condition on $\Gamma$ the following holds 
\begin{equation}\label{eq:estUeps}
\left|U^{\eps}_s( z)\right|^2= \left| u\left(s, {X}^{\eps}_{s-}+\Gamma\left({X}^{\eps}_{s-}, z\right)\right) -u\left(s, {X}^{\eps}_{s-}\right) \right|^2  \leq C(1+ \left|{X}^{\eps}_{s-}\right|^2)|z|^2.
\end{equation}
In conclusion, thanks {to the moment estimate \eqref{eq:growth}}, we obtain 
\begin{equation}\label{eq:estZU}
\begin{aligned} 
& \| Z^\eps-Z^{\eps\eps}\|^2_{\mathbb H^2} + \| U^\eps-U^{\eps\eps}\|^2_{\mathbb H^2_{N^\eps}}\\
& \leq \EE\left[\int_{0}^{T} e^{\rho s}\left|Z^\eps_{s}-Z^{\eps\eps}_{s}\right|^{2}  \de s\right] +\EE\left[\int_{0}^{T} \int_{\mathbb{R}^d} e^{\rho s}\left|U^\eps_s( z)- U^{\eps\eps}_s( z)\right|^{2} \nu^\eps(\de z)  \de s \right]\\ 
&  \leq Ce^{\rho T} (1+|x|^2) \int_{\RR^d} |z|^2 \nu_\eps(\mathrm d z) 
\end{aligned}
\end{equation}
Again from \eqref{eq:ItoFormula}, taking $\tau=T$, we have the following estimates
\begin{align}
\begin{aligned} \sup _{t \in[0, T]} & e^{\rho t}\left|Y^\eps_t-Y^{\epsilon\eps}_t\right|^{2} \\  
\leq &\;2 \int_{0}^{T} e^{\rho s}\left|Y^\eps_{s}-Y^{\epsilon\eps}_{s}\right| \\ 
& \cdot\left|f(s,X^\eps_{s-}, Y^\eps_{s-}, Z^\eps_{s}, \int_{\RR^d} U^\eps_s(z) \nu(\mathrm d z))-f(s,X^\epsilon_{s-} Y^{\epsilon\eps}_{s-}, Z^{\epsilon\eps}_{s}, \int_{\RR^d} \hat U^{\epsilon\eps}_s(z) \nu(\mathrm d z))\right| \de s  \\ 
&+4 \sup _{t \in[0, T]}\left|\int_{0}^{t} e^{\rho s}\left(Y^\eps_{s-}-Y^{\epsilon\eps}_{s-}\right)\left(Z^\eps_s-Z^{\epsilon\eps}_s\right) \de W_s\right| \\
 &+4 \sup _{t \in[0, T]}\left|\int_{0}^{t} \int_{\mathbb{R}^d} e^{\rho s}\left(Y^\eps_{s-}-Y^{\epsilon\eps}_{s-}\right)\left(U^\eps_s( z)-\hat U^{\epsilon\eps}_s( z)\right) \tilde{N}(\de s, \de z)\right| . \end{aligned}
\end{align}
Taking the expectation and using  the Burkholder-Davis-Gundy inequality together with other classical estimates we get
\begin{align}
\begin{aligned}
\mathbb{E}&\left[\sup _{t \in[0, T]} e^{\rho t}\left|Y^\eps_t-Y^{\epsilon\eps}_t\right|^{2} \right]\\
&\leq 2\mathbb{E}\left[\int_{0}^{T} e^{\rho s}\left|Y^\eps_{s}-Y^{\eps\epsilon}_{s}\right|\right. \\ 
& \cdot\left.\left|f(s,X^\eps_{s-}, Y^\eps_{s-}, Z^\eps_{s}, \int_{\RR^d} U^\eps(s,z) \nu(\mathrm d z))-f(s,X^\epsilon_{s-} Y^{\epsilon\eps}_{s-}, Z^{\epsilon\eps}_{s}, \int_{\RR^d} \hat U^{\eps\epsilon}(s,z) \nu(\mathrm d z))\right| \de s\right]\\
&+C \mathbb{E}\left[\left(\int_{0}^{T} e^{2\rho s}\left|Y^\eps_{s-}-Y^{\eps\epsilon}_{s-}\right|^2\left|Z^\eps_s-Z^{\eps\epsilon}_s\right|^2 \de s\right)^{\frac{1}{2}} \right]\\
&+C \mathbb{E}\left[\left(\int_{0}^{T} \int_{\mathbb{R}^d} e^{2\rho s}\left|Y^\eps_{s-}-Y^{\eps\epsilon}_{s-}\right|^2\left|U^\eps_s( z)-\hat U^{\eps\epsilon}_s(z)\right|^2 {N}(\de s, \de z)\right)^{\frac{1}{2}} \right]\\
&\leq 2\mathbb{E}\left[\int_{0}^{T} e^{\rho s}\left|Y^\eps_{s}-Y^{\eps\epsilon}_{s}\right|\right. \\ 
& \cdot\left.\left|f(s,X^\eps_{s-}, Y^\eps_{s-}, Z^\eps_{s}, \int_{\RR^d} U^\eps_s(z) \nu(\mathrm d z))-f(s,X^\epsilon_{s-} Y^{\eps\epsilon}_{s-}, Z^{\eps\epsilon}_{s}, \int_{\RR^d} \hat U^{\eps\epsilon}_s(z) \nu(\mathrm d z))\right| \de s\right]\\
&+C \mathbb{E}\left[\sup _{t \in[0, T]} e^{\frac{\rho}{2} t}\left|Y^\eps_{t}-Y^{\eps\epsilon}_{t}\right|\left(\int_{0}^{T} e^{\rho s}\left|Z^\eps_s-Z^{\eps\epsilon}_s\right|^2 \de s\right)^{\frac{1}{2}} \right]\\
&+C \mathbb{E}\left[\sup _{t \in[0, T]} e^{\frac{\rho}{2} t}\left|Y^\eps_{t}-Y^{\eps\epsilon}_{t}\right|\left(\int_{0}^{T} \int_{\mathbb{R}^d}e^{\rho s} \left|U^\eps_s(z)-\hat U^{\eps\epsilon}_s( z)\right|^2 {N}(\de s, \de z)\right)^{\frac{1}{2}} \right]
\end{aligned}
\end{align}
We use again \eqref{eq:productInequality} to majorate  the two final terms.
\begin{align}
\begin{aligned}
\mathbb{E}&\left[\sup _{t \in[0, T]} e^{\rho t}\left|Y^\eps_t-Y^{\eps\epsilon}_t\right|^{2} \right]\\
&\leq 2\mathbb{E}\left[\int_{0}^{T} e^{\rho s}\left|Y^\eps_{s}-Y^{\epsilon\eps}_{s}\right|\right. \\ 
& \cdot\left.\left|f(s,X^\eps_{s-}, Y^\eps_{s-}, Z^\eps_{s}, \int_{\RR^d} U^\eps_s(z) \nu(\mathrm d z))-f(s,X^\epsilon_{s-} Y^{\eps\epsilon}_{s-}, Z^{\eps\epsilon}_{s}, \int_{\RR^d} \hat U^{\eps\epsilon}_s(z) \nu(\mathrm d z))\right| \de s\right]\\
&+\frac{C}{\tilde\alpha}\mathbb{E}\left[\sup _{t \in[0, T]} e^{\rho t}\left|Y^\eps_t-Y^{\eps\epsilon}_t\right|^{2} \right]\\
&+C\tilde\alpha\mathbb{E}\left[\int_{0}^{T} e^{\rho s}\left|Z^\eps_s-Z^{\eps\epsilon}_s\right|^2 \de s\right]\\
&+C\tilde\alpha \mathbb{E}\left[\int_{0}^{T} \int_{\mathbb{R}^d}e^{\rho s} \left|U^\eps_s(z)-\hat U^{\eps\epsilon}_s( z)\right|^2 {N}(\de s, \de z)\right]
\end{aligned}
\end{align}
In the last expectation we can substitute the random measure with its compensator hence
\begin{align}
\begin{aligned}
\mathbb{E}&\left[\sup _{t \in[0, T]} e^{\rho t}\left|Y^\eps_t-Y^{\eps\epsilon}_t\right|^{2} \right]\\
&\leq 2\mathbb{E}\left[\int_{0}^{T} e^{\rho s}\left|Y^\eps_{s}-Y^{\eps\epsilon}_{s}\right|\right. \\ 
& \cdot\left.\left|f(s,X^\eps_{s-}, Y^\eps_{s-}, Z^\eps_{s}, \int_{\RR^d} U^\eps_s(z) \nu(\mathrm d z))-f(s,X^\epsilon_{s-} Y^{\eps\epsilon}_{s-}, Z^{\eps\epsilon}_{s}, \int_{\RR^d} \hat U^{\eps\epsilon}_s(z) \nu(\mathrm d z))\right| \de s\right]\\
&+\frac{C}{\tilde\alpha}\mathbb{E}\left[\sup _{t \in[0, T]} e^{\rho t}\left|Y^\eps_t-Y^{\eps\epsilon}_t\right|^{2} \right]\\
&+C\tilde\alpha\mathbb{E}\left[\int_{0}^{T} e^{\rho s}\left|Z^\eps_s-Z^{\eps\epsilon}_s\right|^2 \de s\right]\\
&+C\tilde\alpha \mathbb{E}\left[\int_{0}^{T} \int_{\mathbb{R}^d}e^{\rho s} \left|U^\eps_s( z)-\hat U^{\eps\epsilon}_s( z)\right|^2 \nu(\de z)\de s\right]
\end{aligned}
\end{align}
We choose $\tilde\alpha>C$, then we have
\begin{align}
\begin{aligned}
\mathbb{E}&\left[\sup _{t \in[0, T]} e^{\rho t}\left|Y^\eps_t-Y^{\eps\epsilon}_t\right|^{2} \right]\\
&\quad \leq 2C\mathbb{E}\left[\int_{0}^{T} e^{\rho s}\left|Y^\eps_{s}-Y^{\eps\epsilon}_{s}\right|\right. \\ 
& \cdot\left.\left|f(s,X^\eps_{s-}, Y^\eps_{s-}, Z^\eps_{s}, \int_{\RR^d} U^\eps_s(z) \nu(\mathrm d z))-f(s,X^\epsilon_{s-} Y^{\eps\epsilon}_{s-}, Z^{\eps\epsilon}_{s}, \int_{\RR^d} \hat U^{\eps\epsilon}_s(z) \nu(\mathrm d z))\right| \de s\right]\\
&+C\left(\mathbb{E}\left[\int_{0}^{T} e^{\rho s}\left|Z^\eps_s-Z^{\eps\epsilon}_s\right|^2 \de s\right] +  \mathbb{E}\left[\int_{0}^{T} \int_{\mathbb{R}^d}e^{\rho s} \left|U^\eps_s(z)-\hat U^{\eps\epsilon}_s(z)\right|^2 \nu(\de z)\de s\right]\right)
\end{aligned}
\end{align}
Applying again \eqref{eq:productInequality}, for some $\hat \alpha>0$, 
\begin{align*}
& 2 \mathbb{E}\left[\int_{0}^{T} e^{\rho s}\left|Y^\eps_{s}-Y^{\eps\epsilon}_{s}\right|\right. \\ 
&\quad  \cdot\left.\left|f(s,X^\eps_{s-}, Y^\eps_{s-}, Z^\eps_{s}, \int_{\RR^d} U^\eps_s(z) \nu(\mathrm d z))-f(s,X^\epsilon_{s-} Y^{\eps\epsilon}_{s-}, Z^{\eps\epsilon}_{s}, \int_{\RR^d} \hat U^{\eps\epsilon}_s(z) \nu(\mathrm d z))\right| \de s\right]\\
& \leq  \hat \alpha \mathbb{E}\left[\int_{0}^{T} e^{\rho s}\left|Y^\eps_{s}-Y^{\eps\epsilon}_{s}\right|^2\de s\right] \\
& \quad + \frac{1}{\hat\alpha}  \mathbb{E}\left[\int_{0}^{T} e^{\rho s}\left|f(s,X^\eps_{s-}, Y^\eps_{s-}, Z^\eps_{s}, \int_{\RR^d} U^\eps_s(z) \nu(\mathrm d z))-f(s,X^\epsilon_{s-} Y^{\eps\epsilon}_{s-}, Z^{\eps\epsilon}_{s}, \int_{\RR^d} \hat U^{\eps\epsilon}_s(z) \nu(\mathrm d z))\right|^2\de s\right] \\
& \leq  \left(\hat \alpha +  \frac{C}{\hat\alpha}\right) \mathbb{E}\left[\int_{0}^{T} e^{\rho s}\left|Y^\eps_{s}-Y^{\eps\epsilon}_{s}\right|^2\de s\right] \\
& \quad + \frac{C}{\hat\alpha}  \mathbb{E}\left[\int_{0}^{T} e^{\rho s}|Z^\eps_{s} - Z^{\eps\epsilon}_{s}|^2 \de s\right]+ \frac{C}{\hat\alpha}\mathbb E\left[ \int^T_0  \int_{\RR^d}e^{\rho s} |U^\eps_s(z) -\hat U^{\eps\epsilon}_s(z)  |^2\nu(\mathrm d z)\de s\right] 
\end{align*}
Therefore,  we get 
\begin{align*}
\begin{aligned}
\mathbb{E}&\left[\sup _{t \in[0, T]} e^{\rho t}\left|Y^\eps_t-Y^{\eps\epsilon}_t\right|^{2} \right]\\
& \leq C \mathbb{E}\left[\int_{0}^{T} e^{\rho s}\left|Y^\eps_{s}-Y^{\eps\epsilon}_{s}\right|^2\de s\right] + C \mathbb{E}\left[\int_{0}^{T} e^{\rho s}|Z^\eps_{s} - Z^{\eps\epsilon}_{s}|^2 \de s\right]+ C\mathbb E\left[ \int^T_0  \int_{\RR^d}e^{\rho s} |U^\eps_s(z) -\hat U^{\eps\epsilon}_s(z)  |^2\nu(\mathrm d z)\de s\right] 
\end{aligned}
\end{align*}
For controlling the last to terms we use  \eqref{eq:estUeps}-\eqref{eq:estZU} getting
\begin{align*}
\begin{aligned}
& \mathbb{E}\left[\int_{0}^{T} e^{\rho s}\left|Z^\eps_s-Z^{\eps\epsilon}_s\right|^2 \de s\right] +  \mathbb{E}\left[\int_{0}^{T} \int_{\mathbb{R}^d}e^{\rho s} \left|U^\eps_s(z)-\hat U^{\eps\epsilon}_s( z)\right|^2 \nu(\de z)\de s\right]\\
& \leq \mathbb{E}\left[\int_{0}^{T} e^{\rho s}\left|Z^\eps_s-Z^{\eps\epsilon}_s\right|^2 \de s\right] \\
& \quad +  \mathbb{E}\left[\int_{0}^{T} \int_{\mathbb{R}^d}e^{\rho s} \left|U^\eps_s( z)-U^{\eps\epsilon}_s( z)\right|^2 \nu^\eps(\de z)\de s\right] + \mathbb{E}\left[\int_{0}^{T} \int_{\mathbb{R}^d}e^{\rho s} \left|U^\eps_s(z)\right|^2 \nu_\eps(\de z)\de s\right] \\
& \leq C e^{\rho T} (1+|x|^2) \int_{\RR^d} |z|^2 \nu_\eps (\de z). 
\end{aligned}
\end{align*}
Then, one has
\begin{align*}
\mathbb{E}\left[\sup _{t \in[0, T]} e^{\rho t}\left|Y^\eps_t-Y^{\eps\epsilon}_t\right|^{2} \right]\leq C \mathbb{E}\left[\int_{0}^{T} e^{\rho s}\left|Y^\eps_{s}-Y^{\eps\epsilon}_{s}\right|^2\de s\right] + C (1+|x|^2) \int_{\RR^d} |z|^2 \nu_\eps (\de z)
\end{align*}
and applying Gronwall's inequality we obtain 
$$
\|Y^\eps-Y^{\eps\epsilon}\|^2_{\mathbb S^2} \leq  \mathbb{E}\left[\sup _{t \in[0, T]} e^{\rho t}\left|Y^\eps_t-Y^{\eps\epsilon}_t\right|^{2} \right]\leq C (1+|x|^2) \int_{\RR^d} |z|^2 \nu_\eps (\de z)
$$
that together with \eqref{eq:estZU} gives the desired result.
\end{proof}
{
\subsection{On error estimates for the Deep BSDEs solver with jumps}
In order to obtain complete  estimates of the error between the numerical solution provided by our solver and the exact solution $(Y,Z,U)$ of the BSDE \eqref{eq:backward} one should take into account, on top of the error introduced by the approximation of small jumps and estimated in sections \ref{sec:errorforward} and \ref{sec:errorBack}, also the error associated to the ANN approximation  and to the optimization procedure.

Indeed, if we denote by $(\hat Y, \hat{\mathcal Z}, \hat{\mathcal C})$  the output provided by the solver, 
i.e. given  $\hat\Theta\in \RR^R$ the optimal set of parameters resulting from the solution of the optimization of \eqref{eq:SOCP} where $\sigma$ and $\nu$ are replaced by $\sigma_\eps:=\sigma + \sqrt{\Sigma_\eps}^\top \gamma$ and $\nu^\eps$, respectively, we take 
for $n=0,\ldots, M$, $\hat Y_n  = \tilde Y^{\hat\Theta}_n$, $\hat{\mathcal Z}_n  = \mathcal{Z}^{\hat\Theta}_n$ and $\hat {\mathcal C}_n = {\mathcal C}^{\hat \Theta}_n$,
we clearly have 
\begin{align*}
& \underset{n=0,\ldots, M-1}\max\EE\left[\underset{t\in [t_n,t_{n+1}]}\sup |Y_t-\hat Y_n|^2\right] + \sum^{M-1}_{n=0}\EE\left[\int^{t_{n+1}}_{t_n}|Z_t- \hat{\mathcal Z}_n|^2\de t\right] \\
& \quad +  \sum^{M-1}_{n=0}\EE\left[\left|\int^{t_{n+1}}_{t_n}\int_{\RR^d} U_t(z)\nu^\eps(\de z)\de t - \hat{\mathcal C}_n \Delta t\right|^2\right]\\
& \leq C\left(\left\|Y-Y^{\eps\eps}\right\|_{\mathbb{S}^{2}}^{2} + \|Z- Z^{\eps\eps}\|^2_{\mathbb H^2} +  \|U- U^{\eps\eps}\|^2_{\mathbb H^2_{N^\eps}}\right)\\
& \quad+ 2 \bigg( \underset{n=0,\ldots, M-1}\max\EE\left[\underset{t\in [t_n,t_{n+1}]}\sup |Y^{\eps\eps}_t-\hat Y_n|^2\right] + \sum^{M-1}_{n=0}\EE\left[\int^{t_{n+1}}_{t_n}|Z^{\eps\eps}_t- \hat {\mathcal Z}_n|^2\de t\right]\\
& \quad+  \sum^{M-1}_{n=0}\EE\left[\left|\int^{t_{n+1}}_{t_n}\int_{\RR^d} U^{\eps\eps}_t(z)\nu^\eps(\de z)\de t - \hat{\mathcal C}_n \Delta t\right|^2\right]\bigg)
\end{align*}

The terms in the first brackets of the right hand side are bounded by Theorem \ref{teo:errorback}. We are currently working to establish some bounds for the terms in the second brackets. We conjecture that results in this direction can be obtained by extending  the techniques in \cite{HanLon18} to the case with jumps.}

\color{black}
\section{Numerical results}\label{sec:numerics}
To evaluate the performance of our algorithm, we present four examples where we compare the results either from a theoretical perspective or through Monte Carlo simulations. { All Monte Carlo simulations, used to generate reference solutions are based on $2^{16}$ sample paths. Moreover, the weight in \eqref{eq:SOCP} is set to $K=\frac{T}{M}\times0.1$, for all experiments. Here we recall that $M+1$ is the number of temporal grid points.} The code for our experiments is available at \url{https://github.com/AlessandroGnoatto/DeepBsdeSolverWithJumps}. In the context of applying our methodology to option pricing, we impose additional assumptions on the L\'evy measure to ensure the integrability of the underlying asset. Specifically, we require the driving L\'evy process to have finite exponential moments.

\begin{itemize}
\item[(A6)] There exists a constant $\cM>1$ such that 
\begin{align}
\int_{\RR^d}e^{u^\top z}\nu(\de z)<\infty,
\end{align}
for all $u\in[-\cM,\cM]^d$.
\end{itemize}

\subsection{Pure jump expectation}
Consider the following pure jump process
\begin{equation}\label{eq:pureJumpDif}
    \frac{\de X_t}{X_{t-}} = \int_{\mathbb{R}}(e^z-1)\tilde N(\de z, \de t)\; ,
\end{equation}
where $\nu(\de z) = \lambda \varphi(z)\de z$ is the Lévy measure with $\lambda>0$ and $\varphi(z)=\frac{1}{\sqrt{2\pi}\sigma_J} e^{-\frac{1}{2}\left(\frac{z-\mu_J}{\sigma_J}\right)^2}$.

%
%
The solution to \eqref{eq:pureJumpDif} satisfies
\begin{equation}
    X_T = X_t \exp\left\lbrace -\int_{\mathbb{R}}(e^z-1-z)\nu(\de z) (T-t)+ \int_t^T\int_{\mathbb{R}}z\tilde N(\de z, \de s)\right\rbrace \; .
\end{equation}

The characteristic function of $\ln X_T$ is
\begin{equation}\label{eq:exp}
\mathbb{E}\left[e^{iq\ln X_T}|\mathcal{F}_t \right] = \exp\left\lbrace iq\ln X_t -iq(T-t)\psi(-i) + (T-t)\psi(q)\right\rbrace \; ,
\end{equation}
where $\psi(q)\coloneqq \lambda\left(e^{iu\mu_J -\frac{1}{2}q^2\sigma_J^2} -1 \right)$ is the characteristic exponent.

Thanks to the Markov property we can write
\begin{equation}
    \mathbb{E}\left[e^{iq\ln X_T}|\mathcal{F}_t \right] = u(t,X_t) \; ,
\end{equation}
%
where $u$ solves the following PIDE: 
\begin{equation}\label{eq:pidePureJump}
\begin{cases}
u_t(t,x) + \int_{\mathbb{R}} u(t,xe^z)-u(t,x)  -x(e^z-1)u_x(t,x) \nu(\de z) = 0\\
u(T,x) = e^{iq\ln x} \; .
\end{cases}
\end{equation}

%

Therefore, the related BSDE is
\begin{equation}
\begin{cases}
-\de Y_s = - \int_{\mathbb{R}} U_s(z) \tilde N(\de z, \de s)\\
Y_T = e^{iq\ln X_T} \; ,
\end{cases}
\end{equation}
where, by the Feynman-Kac representation formula, $Y_t=u(t,X_t)$ and $U_t(z)=u(t,X_{t^-}e^z)-u(t,X_{t^-})$.

Assuming that $q=-i$ in Equation \eqref{eq:exp}, then 
\begin{equation}
Y_t=\mathbb{E}\left[e^{iq\ln X_T}|\mathcal{F}_t \right] = X_t \; , \; \forall t\in [0,T] \; .
\end{equation}
{
In this first example we set  $X_0=1$, $\lambda=0.3$, $\mu_J=0.5$, $\sigma_J=0.25$. The ANNs have $\mathcal L-1=2$ hidden layers, each with $\upsilon=10$ nodes. We  consider $M = 40$ time steps and perform $10,000$ batch iterations with batch size $2^{10}=1024$, \textit{i.e.,} each batch consists of $2^{10}$ samples and the algorithm runs through $10,000$ batches in total. The high number of paths is used to ensure that the algorithm has converged, but we can empirically observe that we reach a high level of accuracy with significantly fewer training samples. The results are reported in Table \ref{tab:ExpRes} and in Figure \ref{fig:ExpRes}.

\begin{table}[htbp]	
\caption{Numerical results for the pure jump expectation example}
\centering
\begin{tabular}{rccc}
\toprule
Step	&     $\widetilde{Y}_0^\Theta\ (Y_0=1)$	&$\frac{T}{M}\sum_{n=1}^{M-1}\mathbbm{E}|Y_{t_n}-\widetilde{Y}^\Theta_{n}|^2$	& Time (s)\\
\midrule
\midrule
100		& 0.5778 & 7.904e-2	&	96	\\
500		& 0.800 & 1.808e-2	&	159	\\
1000	& 0.955 & 5.232e-4	&	228	\\
2000	& 1.000 & 8.750e-5	&	542	\\
4000	& 0.997 & 2.709e-5&	873	\\
6000	& 1.000 & 4.572e-5	&	1183	\\
8000	& 1.000 & 4.417e-5	&	1496	\\
10000	& 1.000	& 1.19e-5 &	 1617	\\
\bottomrule
\end{tabular}
\label{tab:ExpRes}
\end{table} 

\begin{figure}[htp]
\centering
\begin{tabular}{cc}
          \includegraphics[width=70mm]{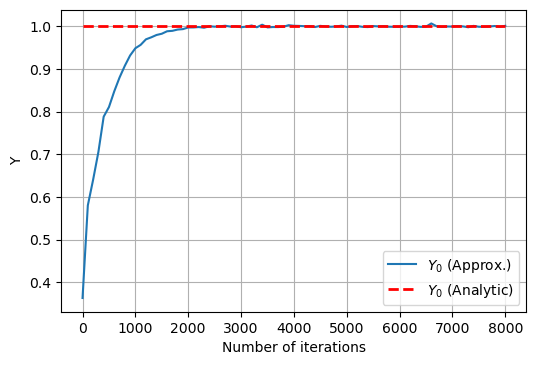}&    
          \includegraphics[width=70mm]{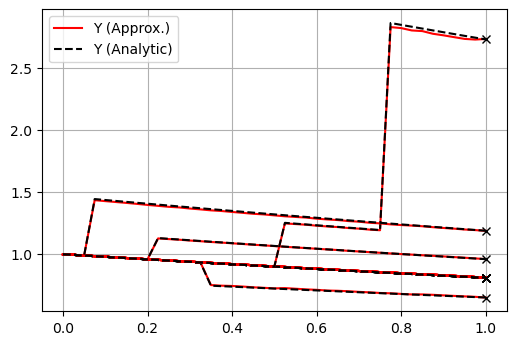}
\end{tabular}
\caption{{\textbf{Left:} Approximate initial value of the BSDE against the number of training steps (the number of batch iterations). \textbf{Right:} 20 representative paths of our approximations compared to the analytical solutions. Note that only four paths contain at least one jump. }}\label{fig:ExpRes}
\end{figure}
}
\subsection{Call option example}\label{sect:calloption}
Let $X$ denote the price of an underlying asset under the risk-neutral probability measure $\mathbb{Q}$ described by the following dynamics
\begin{equation}\label{eq:callDiff}
    \frac{\de X_t}{X_{t-}} = r\de t +\sigma \de W^\mathbb{Q}_t + \int_{\mathbb{R}}(e^z-1)\tilde N(\de z, \de t)\; , \quad  X_0 = x_0 \in \mathbb{R} ,
\end{equation}
where $r\in\mathbb{R}$, $\sigma\in\mathbb{R^+}$ and $\nu(\de z) = \lambda \varphi(z)\de z$ is the Lévy measure with $\lambda>0$ and $\varphi(z)=\frac{1}{\sqrt{2\pi}\sigma_J} e^{-\frac{1}{2}\left(\frac{z-\mu_J}{\sigma_J}\right)^2}$.

For $t,s\in[0,T]$ with $t\geq s$, the solution to \eqref{eq:callDiff} is given by:
\begin{align*}
    X_t &= X_s \exp\left\lbrace \left(r-\frac{\sigma^2}{2} \right)(t-s)+\sigma (W_t^\mathbb{Q}-W_s^\mathbb{Q})   -\int_{\mathbb{R}}(e^z-1-z)\nu(\de z) (t-s)+ \int_s^t\int_{\mathbb{R}}z\tilde N(\de z, \de u)\right\rbrace \; .
\end{align*}

Let $Y$ be a European call option on $X$ with value at time $t\in [0,T]$ given by 
\begin{equation*}
Y_t = \mathbb{E}^\mathbb{Q}\left[e^{-r(T-t)}(X_T-k)^+ |\cF_t \right]\; ,
\end{equation*}
where $k\in\mathbb{R^+}$ is the strike price and $(x)^+=\max(x,0)$. Thanks to the Markov property of the process $X$ we can write $Y_t = u(t,X_t)$, where $u$ solves the following PIDE: 
\begin{equation*}
\begin{cases}
u_t(t,x) + u_x(t,x)rx + \frac{1}{2}\sigma^2x^2u_{xx}(t,x)   + \int_{\mathbb{R}} u(t,xe^z)-u(t,x)   -x(e^z-1)u_x(t,x) \nu(\de z) - ru(t,x) = 0\\
u(T,x) = (x-k)^+ \; .
\end{cases}
\end{equation*}

Therefore, the related BSDE is
\begin{equation*}
\begin{cases}
-\de Y_s = -rY_s \de s - Z_s \de W_s^\mathbb{Q} - \int_{\mathbb{R}} U(s,z) \tilde N(\de z, \de s)\\
Y_T = (X_T-k)^+ \; ,
\end{cases}
\end{equation*}
where, from the Feynman-Kac representation formula we also have $Z_t = u_x(t,X_{t^-})\sigma X_{t^-}$ and $U_t(z)=u(t,X_{t^-}e^z)-u(t,X_{t^-})$.

 We aim to accurately approximate the paths of the option price by using the deep solver and compare them with the semi-analytical solution presented in \cite{merton1976option}. {We set  $X_0=1$, $k=0.9$, $r=0.05$, $\sigma=0.25$, $\lambda=0.3$, $\mu_J=0.5$, $\sigma_J=0.25$. For the numerical scheme, we use $M=40$ time steps and $10,000$ batch iterations with batch size $2^{10}=1024$. We use ANNs with $2$ hidden layers each one with $\upsilon = 10$ nodes. The numerical results are reported in Table \ref{tab:CallRes} and Figure \ref{fig:CallRes}. As a reference, the semi-analytic initial value is given by $Y_0=0.2273$.

Figure \ref{fig:CallRes} provides a pathwise comparison of the semi-analytical solution (dotted line) against the approximation provided by the solver. On a path by path basis, the solver provides a good approximation. We remark that the quality decreases when the semi-analytical price is close to zero: in such cases we see that the path predicted by the solver can become negative. This is a common issue which typically arises in deep BSDE methods which relies on an Euler--Maruyama approximation of the $Y-$component. The problem can become more pronounced in the presence of jumps, as large negative jumps can cause the approximate solution to jump into negative territory.}

\begin{table}[htbp]	
\caption{{Numerical results for the single asset call option example.}}
\centering
\begin{tabular}{rccc}
\toprule
Step	&     $\widetilde{Y}_0^\Theta\ (Y_0=0.2273)$		&$\frac{T}{M}\sum_{n=1}^{M-1}\mathbbm{E}|Y_{t_n}-\widetilde{Y}^\Theta_{n}|^2$	& Time (s)\\
\midrule
\midrule
100		& 0.141 & 9.10e-2 	&	93	\\
500		& 0.189 &  4.47e-3 	&	153	\\
1000	& 0.222 & 4.49e-3 	& 		215	\\
2000	& 0.231  &  4.47e-3	&	352 \\
4000	& 0.228 & 4.36e-3 &	632 \\
6000	& 0.227 &  4.05e-3 & 1111	\\
8000	& 0.228 &  7.62e-3	& 1386	\\
10000	&  0.227 & 4.21e-3  & 1664	\\
\bottomrule
\end{tabular}
\label{tab:CallRes}
\end{table} 

\begin{figure}[htp]
\centering
\begin{tabular}{cc}
          \includegraphics[width=70mm]{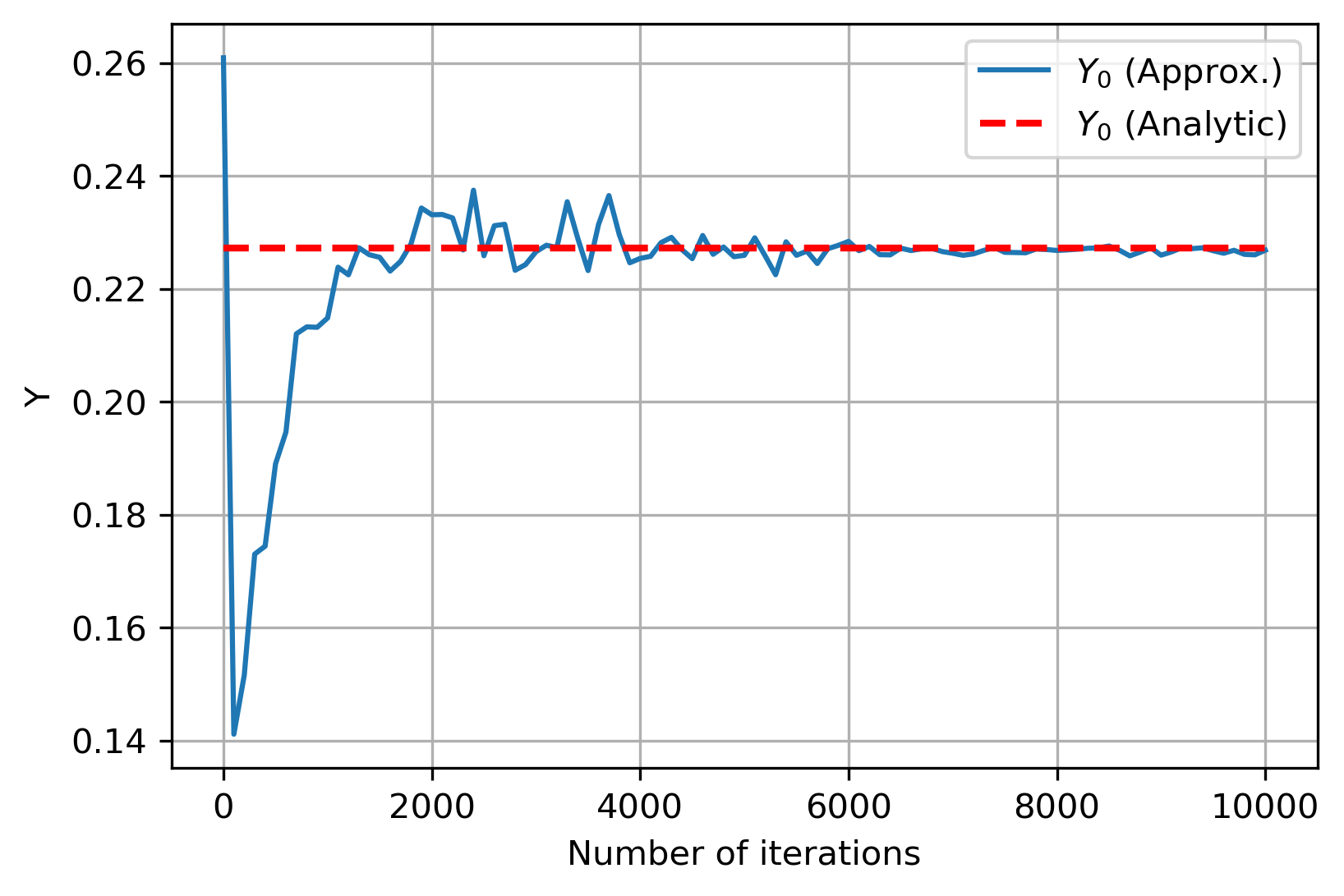}&    
          \includegraphics[width=70mm]{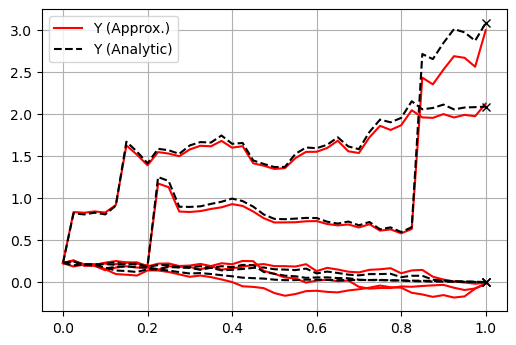}
\end{tabular}
\caption{{\textbf{Left:} Approximate initial value of the BSDE against the number of training steps (the number of batch iterations). \textbf{Right:}  Five representative paths of our approximations compared to the analytical solutions.}}\label{fig:CallRes}
\end{figure}
\begin{remark}
Clearly, $Y_0$ could be accurately approximated via a Monte Carlo method in significantly less time (even in high-dimensional problems) than our algorithms. However, the comparison is not fair since our method approximates the solution of the PIDE along the stochastic paths $(t,X_t)_{t\in[0,T]}$ and not only at $(0,x)$. Pathwise approximate solutions could be obtained via the Monte Carlo method by resorting to nested simulations, which are notoriously expensive.
\end{remark}

\subsection{A Basket call option example}
The examples we presented so far are mainly meant to provide a validation of the methodology in the one dimensional case. The example we proceed to discuss now instead provides evidence that the methodology is viable in a high dimensional setting. We consider the case of several underlying assets $(X^1;\cdots;X^d)$: 
\begin{equation}\label{eq:basketDiff}
    \frac{\de X_t^i}{X_{t-}^i} = r^i\de t +\sigma^i \de W^{\mathbb{Q},i} + \int_{\mathbb{R}}(e^z-1)\tilde N^i(\de z, \de t)\; , \quad  X_0^i = x_0^i \in \mathbb{R}^d , \quad i=1,\cdots , d \, ,
\end{equation}
where $r^i\in\mathbb{R}$, $\sigma^i\in\mathbb{R^+}, i=1,\cdots,d$, $W^{\mathbb{Q}} = (W^{\mathbb{Q},1},\cdots ,W^{\mathbb{Q},d})$ is a standard Brownian motion in $\mathbb{R}^d$  and $\nu^i(\de z)=\lambda \varphi(z)\de z$, $i=1,\cdots,d$, is the Lévy measure with $\lambda>0$ and $\varphi(z)=\frac{1}{\sqrt{2\pi}\sigma_J} e^{-\frac{1}{2}\left(\frac{z-\mu_J}{\sigma_J}\right)^2}$.

The solution to \eqref{eq:basketDiff} satisfies, for any $i=1,\ldots, d$,
\begin{align*}
\begin{aligned}
    X_T^i = X_t^i \exp&\left\lbrace \left(r^i-\frac{\sigma^{i,2}}{2} \right)(T-t)+\sigma^i (W_T^{\mathbb{Q},i}-W_t^{\mathbb{Q},i})\right.\\
    &\left.
   -\int_{\mathbb{R}}(e^z-1-z)\nu^i(\de z) (T-t)+ \int_t^T\int_{\mathbb{R}}z\tilde N^i(\de z, \de s)\right\rbrace \; .
 \end{aligned}
\end{align*}

Let $Y$ be a basket of European call option on $(X^1;\cdots;X^d)$ with value at time $t\in [0,T]$ given by 
\begin{equation}
Y_t = \mathbb{E}^\mathbb{Q}\left[e^{-r(T-t)}\left( \frac{1}{d}\sum_{i=1}^d X_T^i-k\right)^+ \biggr\vert\cF_t \right]\; .
\end{equation}
where $k\in\mathbb{R^+}$ is the strike price. From the Feynman-Kac representation formula we have $Y_t = u(t,X_t^1,\cdots,X_t^d)= u(t,X_t)$ where $u$ solves the following PIDE: 
\begin{equation}\label{eq:pidePureJump2}
\begin{cases}
\!\begin{aligned}
u_t(t,x) + \nabla_x u(t,x)b_x + \frac{1}{2} \tr\left[\sigma\sigma^\intercal(t,x)D^2_xu(t,x) \right]& - ru(t,x) \\
 + \int_{\mathbb{R}^d} u(t,x_{.}e^z)-u(t,x) & -x_{.}(e^z-\mathbf{1})\nabla_x u(t,x) \nu(\de z)  = 0
\end{aligned}\\
u(T,x) = \left( \frac{1}{d}\sum_{i=1}^d x_i -k\right)^+,
\end{cases}
\end{equation}
where we use $_{.}$ to denote the component-wise product between vectors, $e^z$ is the vector where each entry is of the form $e^{z_i}$, $i =1,\ldots d$, $\mathbf{1}$ is the vector with all elements equal to one and $\nu$ is the product of the $d$ L\'evy measures of the individual driving processes.
Therefore, the related BSDE is
\begin{equation}
\begin{cases}
-\de Y_s = -rY_s \de s - Z_s \de W_s^\mathbb{Q} - \int_{\mathbb{R}^d} U(s,z) \tilde N(\de z, \de s)\\
Y_T = \left( \frac{1}{d}\sum_{i=1}^d X_T^i-k\right)^+,
\end{cases}
\end{equation}
where  $U_t(z)=u(t,X_{t^-} {}_{.}e^z)-u(t,X_{t^-})$.

For the basket option, we approximate $Y_0$ using the deep solver. Our approximation is compared with a reference solution computed with a Monte Carlo approximation. {To compute reference solutions of $(Y_t)_{t\in(0,T]}$ we resort to a nested Monte Carlo approximation. We set  $X_0=\mathbf 1$, ${k=0.9}, r=0.05, \sigma^i = \sigma=0.25$ for $i = 1,\ldots,d$, $\lambda=0.3$, $\mu_J=0.5$ and $\sigma_J=0.25$. We consider $M = 40$ time steps and perform $10,000$ batch iterations for $d=5$ and $d=25$, and $20,000$ batch iterations for $d=100$, with  ANNs with  $\mathcal L-1= 2$ hidden layers, each with $\upsilon = 25$, $\upsilon=50$ and $\upsilon=100$ nodes, for $d=5$, $d=25$ and $d=100$, respectively. The batch size is set to $2^{10}=1024$ for $d=5$, $2^8=256$ for $d=25$ and $d=2^6=64$ for $d=100$. The results are displayed in Figure \ref{fig:BasketCallRes}.

\begin{figure}[htp]
\centering
\begin{tabular}{cc}
          \includegraphics[width=70mm]{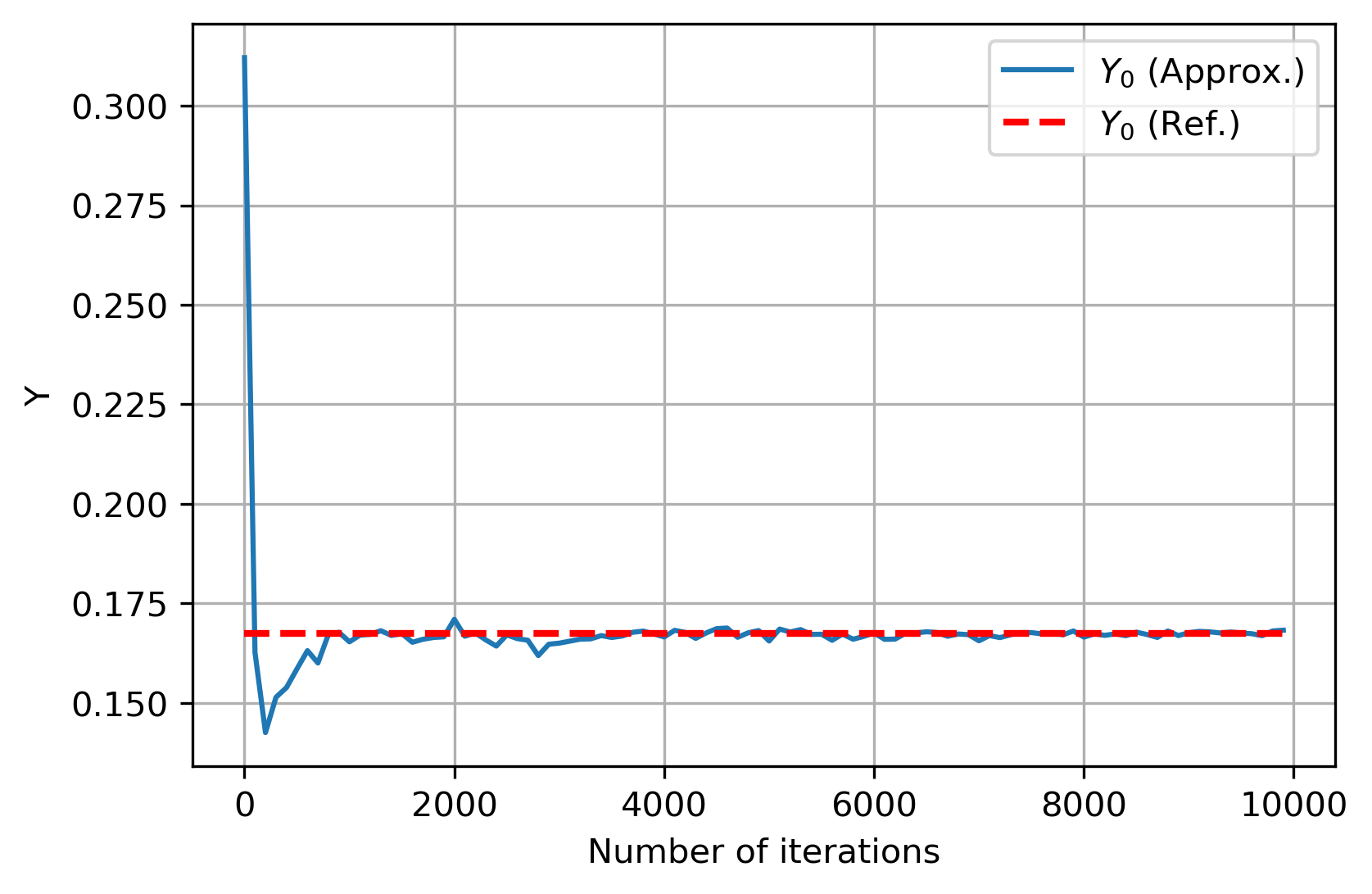}&    
          \includegraphics[width=70mm]{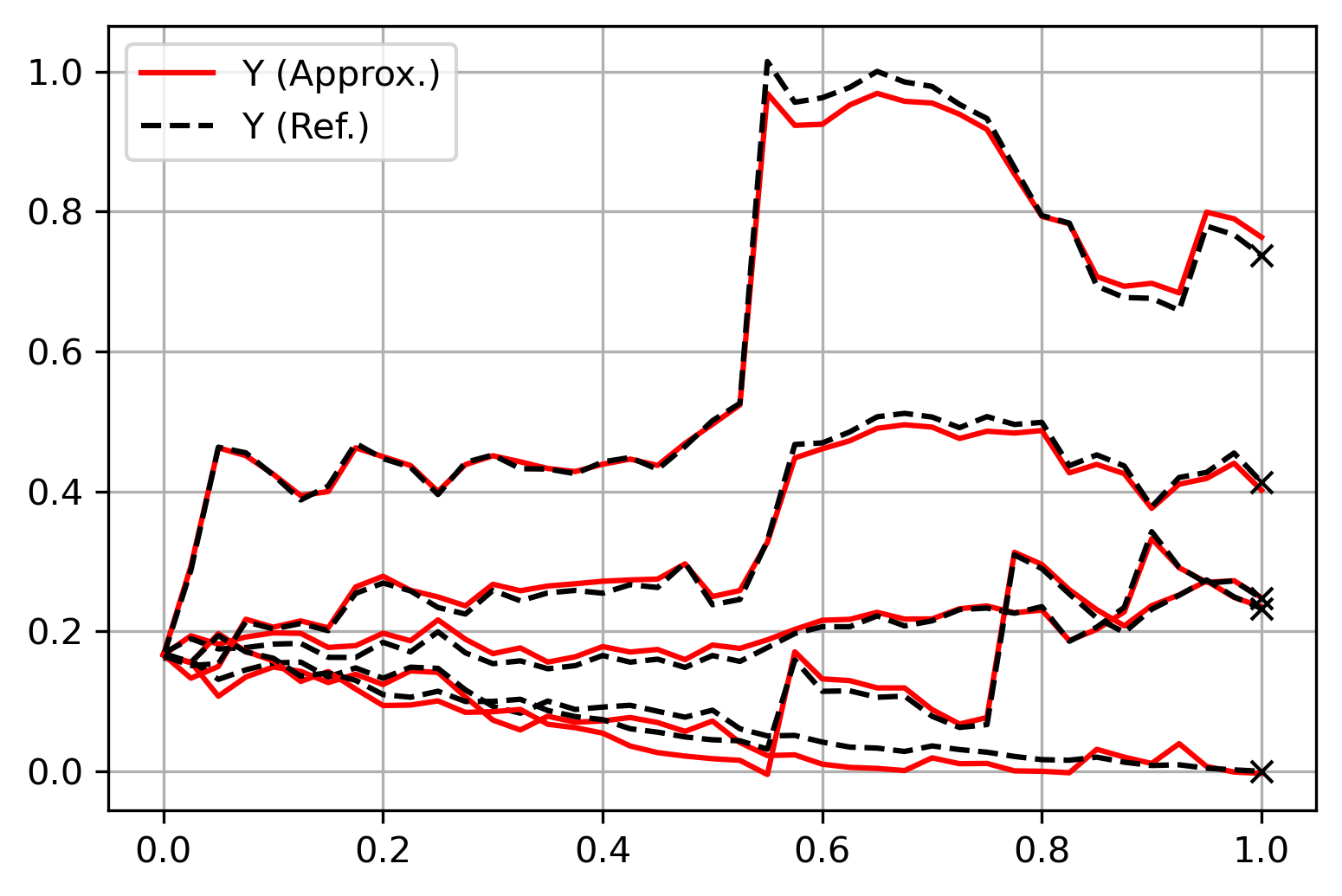}\\
            \includegraphics[width=70mm]{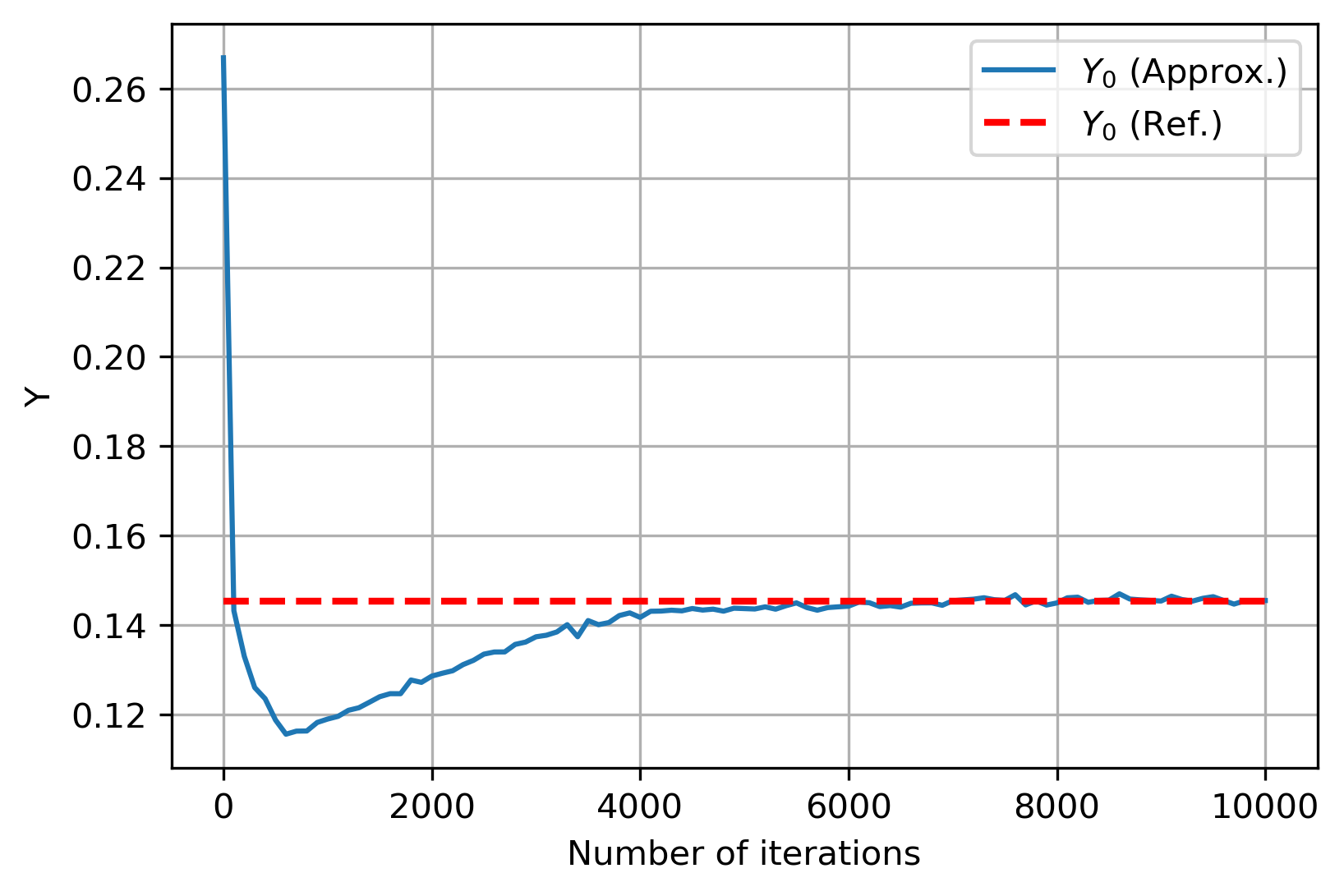}&    
          \includegraphics[width=70mm]{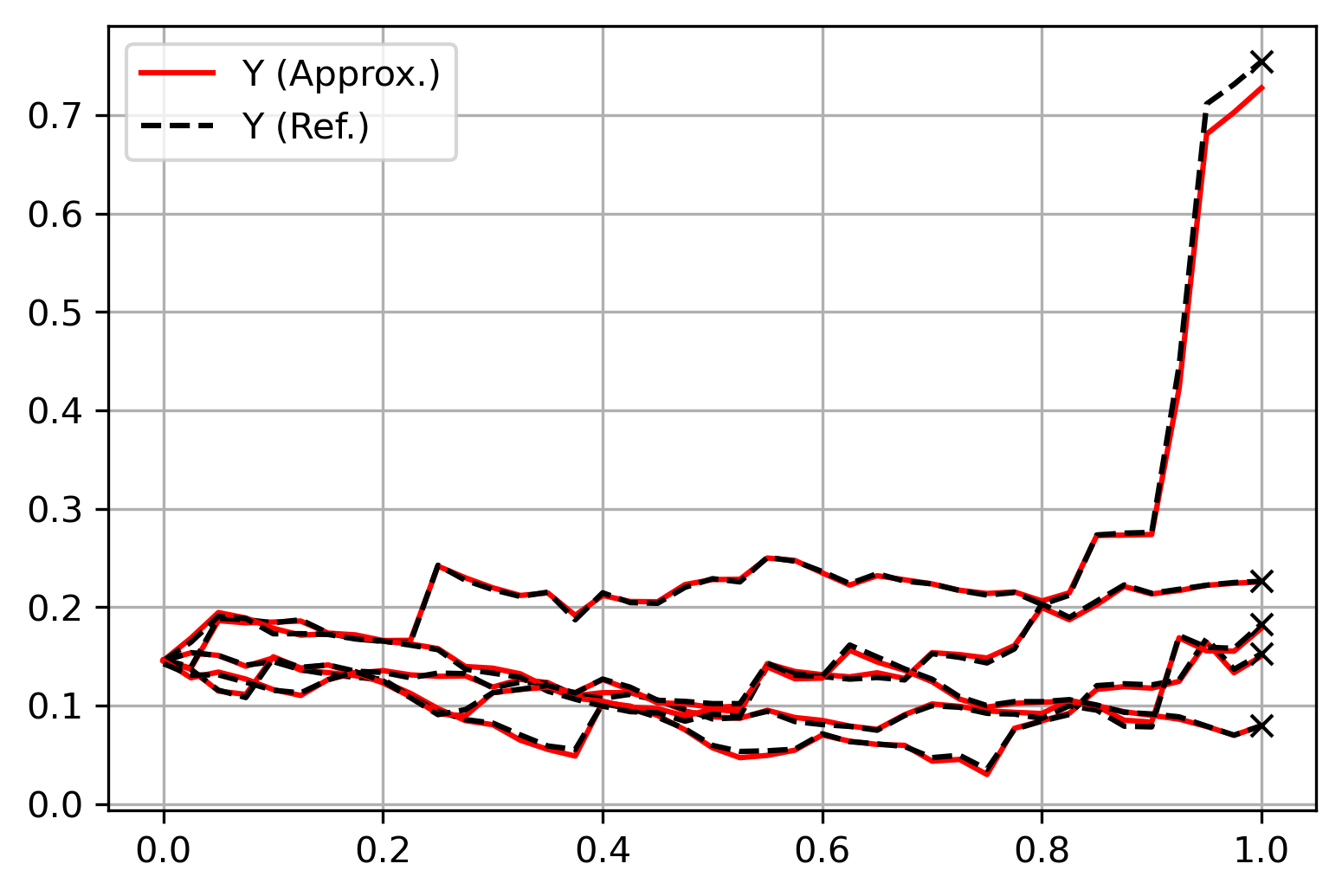}\\
        \includegraphics[width=70mm]{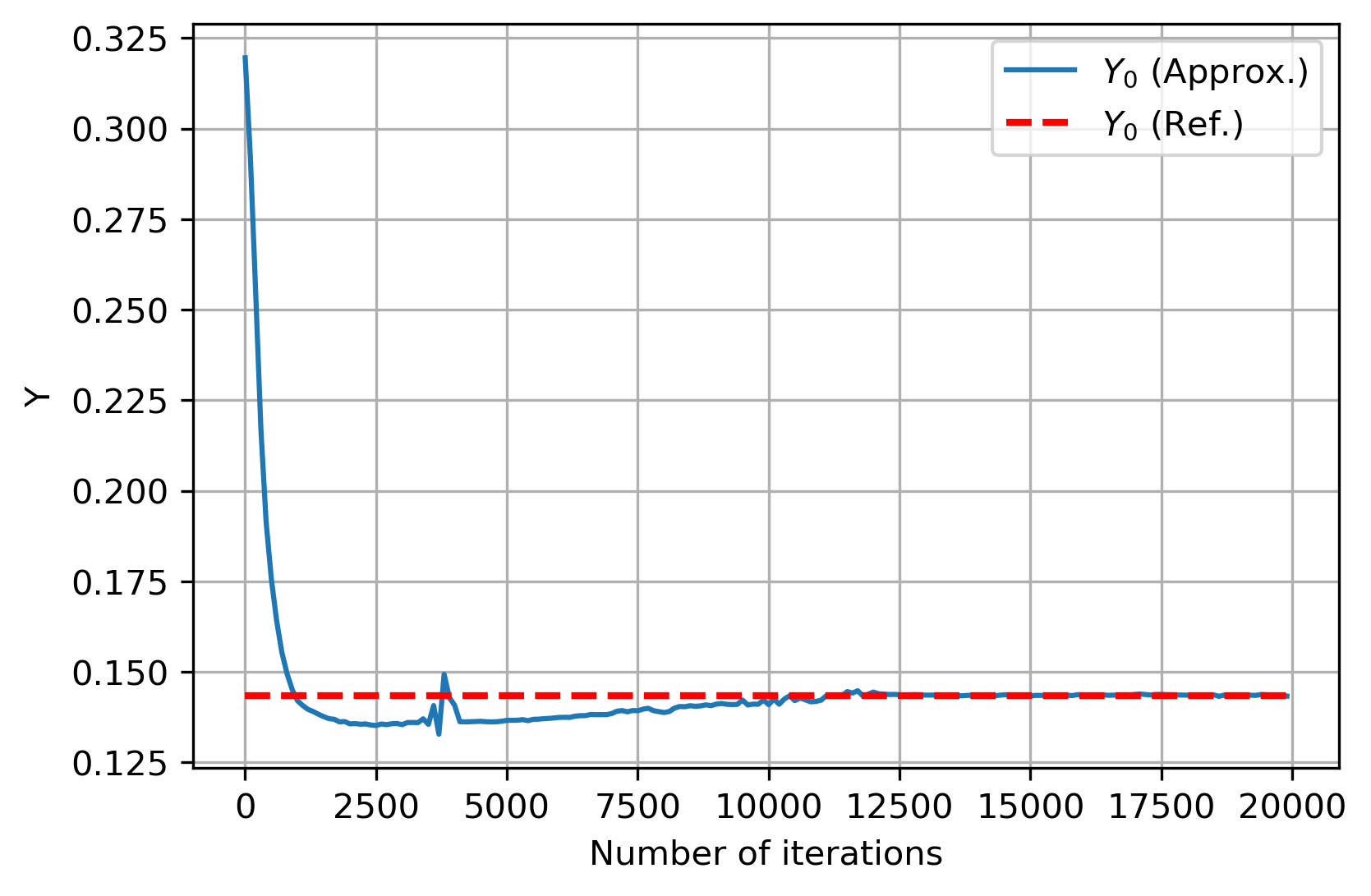}&   \includegraphics[width=70mm]{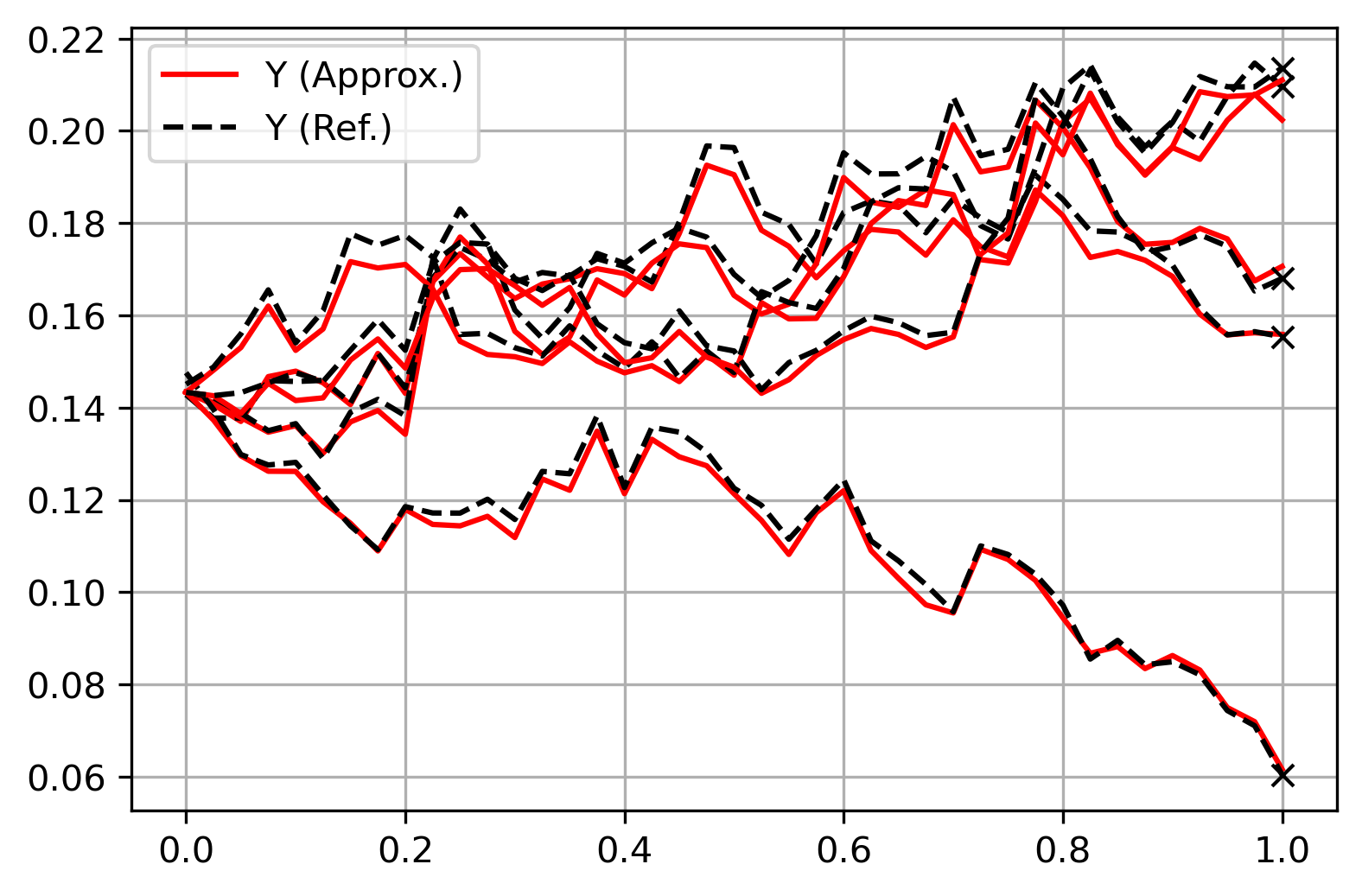}
\end{tabular}
\caption{{\textbf{Left:} Approximate initial value of the BSDE for $d=5$ (upper), $d=25$ (middle) and $d=100$ (lower) against the number of training steps (the number of batch iterations). \textbf{Right:} Five representative paths of our approximations compared to reference solutions obtained from a nested Monte Carlo approximation for $d=5$ (upper), $d=25$ (middle) and $d=100$ (lower). }}\label{fig:BasketCallRes}
\end{figure}

Finally, we vary the intensity jump parameter $\lambda$, and rerun the five dimensional experiment for $\lambda\in\{0.1,0.3,0.5,0.7,0.9,1.1,1.3\}$. In Figure \ref{fig:densityComp} we display the effect of the choice of different intensity parameters on the tail distribution of log-prices distributions. 
\begin{figure}[htbp]
\centering
\includegraphics[scale=0.5]{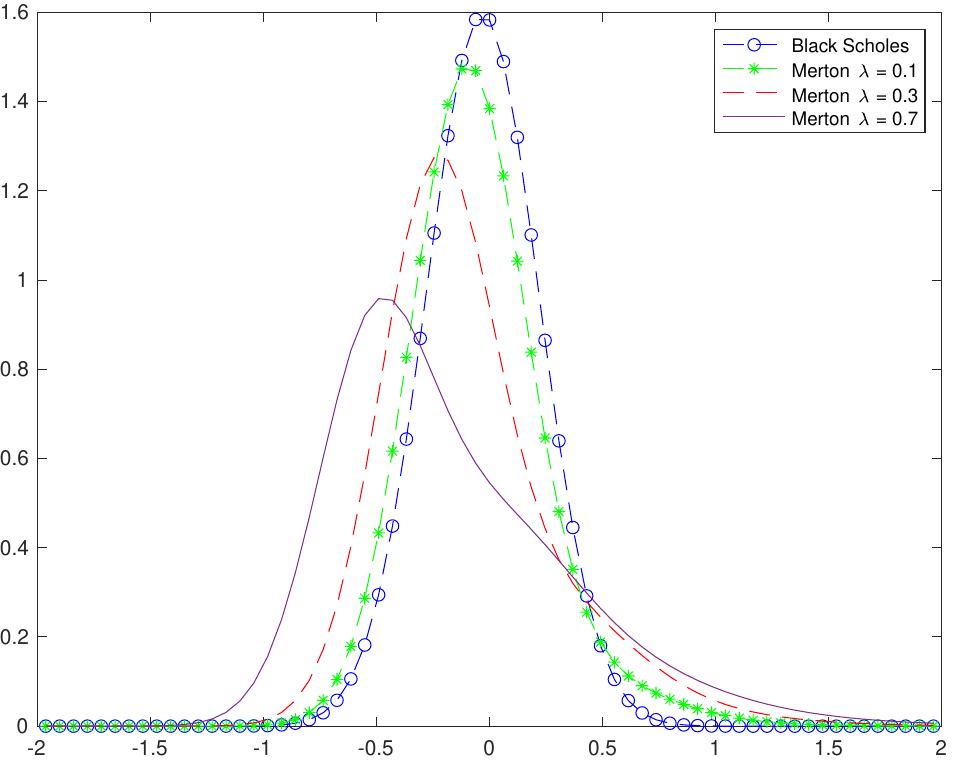}
\caption{{Comparison of log-price densities of a single asset for different values of $\lambda$.}}
\label{fig:densityComp}
\end{figure}
The intuition behind the experiment is the following: if the jump intensity/activity is high, a finer time discretization may be needed in order to accurately approximate the trajectories of the FBSDE. From the presentation of our proposed algorithm, it is clear that it allows for multiple jumps between adjacent time points. However, it remains to be investigated how this affects the accuracy of the algorithm in practice. In Table \ref{tab:SimResLambda_high} we compare our approximate initial values of the BSDE with reference solutions computed with Monte Carlo sampling. Moreover, We include the probability of having more than one jump between adjacent time points which is given by $1 - \big(\text{e}^{-\lambda  \Delta t} (1+\lambda\Delta t)\big)^{dM}$ for each realization of the asset process $\widetilde{X}$. As we can see in Table \ref{tab:SimResLambda_high}, we can obtain high accuracy of the approximate initial value even with jump intensities which implies that we almost have a $10\%$ probability of having at least two jumps between adjacent time points somewhere in each sample.} We remark however that the values of $\lambda$ we are employing in the present experiment are much higher in comparison to those obtained from a calibration to prices of financial products, see among others \citet{kienitz2013financial}.

\begin{table}[htbp]	
\caption{{Simulation results for different values of $\lambda$ for the basked option with $d=5$.}}
\centering
\begin{tabular}{ccccc}
\toprule
$\lambda$ & $\mathbbm{P}(\text{\#jumps}>1)$&	$Y_0$ (MC-approx.)  & $\widetilde{Y}_0^\Theta$	& Relative error  \\
\midrule\midrule
0.1	& 0.062\% &	0.1526	&	0.1528	&	0.15\%			\\
0.3	& 0.56\% &	0.1678	&	0.1672	&	0.40\%			\\
0.5	& 1.5\% &	0.1836	&	0.1826	&	0.59\%			\\
0.7	& 3.0\% &	0.1974	&	0.1988	&	0.60\%			\\
0.9	& 4.9\% &	0.212	&	0.2138	&	0.83\%			\\
1.1	& 7.2\% &	0.2246	&	0.2252	&	0.27\%			\\
1.3	& 9.8\% &	0.2354	&	0.2320	&	1.44\%			\\
\bottomrule
\end{tabular}
\label{tab:SimResLambda_high}
\end{table}


%

\subsection{Infinite activity example}

The final example we propose provides evidence for the feasibility of the proposed algorithm for the infinite activity case. We consider the CGMY process, which is a pure jump L\'evy process $L$ with characteristic triplet $(0,0,\nu(\de z))$, where the L\'evy measure is given by
\begin{align*}
\nu(\de z)=C\frac{e^{-G|z|}}{|z|^{1+Y}}\mathtt{1}_{\{z<0\}}\de z+C\frac{e^{-Mz}}{z^{1+Y}}\mathtt{1}_{\{z>0\}}\de z.
\end{align*}
The L\'evy measure corresponds to that of a difference of two tempered stable subordinators, meaning that we can write $L_t=L^+_t-L^-_t$ where 
\begin{itemize}
\item $L^+$ has triplet $(0,0,\nu^+(\de z))$, where $\nu^+(\de z):=C\frac{e^{-Mz}}{z^{1+Y}}\mathtt{1}_{\{z>0\}}\de z$;
\item $L^-$ has triplet $(0,0,\nu^-(\de z))$, where $\nu^-(\de z):=C\frac{e^{-Gz}}{z^{1+Y}}\mathtt{1}_{\{z>0\}}\de z$.
\end{itemize}
In the following we outline the procedure to approximate a tempered stable subordinator by following \cite{Cont2003} and references therein among others. Since the procedure is the same for $L^+$ and $L^-$, we choose to concentrate on $L^+$. We fix $\epsilon>0$, which represents a threshold for jump sizes where all jumps smaller than $\epsilon$ are approximated by a diffusion process, whereas all jumps larger than $\epsilon$ is approximated by a suitably constructed compensated compound Poisson process. We first look at the large jumps. The jump compensator of the approximating compound Poisson process is given by
\begin{align*}
\int_\epsilon^\infty z C\frac{e^{-Mz}}{z^{1+Y}}\de z&=C\int_\epsilon^\infty z^{-Y} e^{-Mz}\de z=C\int_{M\epsilon}^\infty \left(\frac{\ell}{M}\right)^{-Y} e^{-\ell}\de \ell\\
&=\frac{C}{M^{a}}\int_{M\epsilon}^\infty \ell^{a-1}e^{-\ell} \de \ell \; ,
\end{align*}
where we set $a:=1-Y$. Let us now recall the lower incomplete Gamma function as implemented in  \texttt{scipy} or \texttt{Matlab} $\Gamma(a,M\epsilon):=\frac{1}{\Gamma(a)}\int_0^{M\epsilon}\ell^{a-1}e^{-\ell}\de \ell$, where $\Gamma(a)$ is the Gamma function computed in $a$. We conclude that 
\begin{align*}
\int_\epsilon^\infty z C\frac{e^{-Mz}}{z^{1+Y}}\de z&={C}\Gamma(a)\left(1-\Gamma(a,M\epsilon)\right)\eqcolon-b_{\epsilon^+} \; .
\end{align*}
The intensity of the approximating compound Poisson process is given by
\begin{align*}
\lambda^+_\epsilon:=\int_\epsilon^\infty C\frac{e^{-Mz}}{z^{1+Y}}\de z \; .
\end{align*}
To find a convenient expression for the intensity, we integrate by parts and obtain
\begin{align*}
\int_\epsilon^\infty zC\frac{e^{-Mz}}{z^{1+Y}}\de z=\frac{C}{M}e^{-M\epsilon}\epsilon^{-Y}-\frac{YC}{M}\int_\epsilon^\infty e^{-Mz}z^{-1-Y}\de z \; .
\end{align*}
By rearranging terms, we can write
\begin{align*}
\lambda^+_\epsilon=Ce^{-M\epsilon}\frac{\epsilon^{-Y}}{Y}+\frac{M}{Y}b_{\epsilon^+} \; .
\end{align*}
Next, we need to determine the jump size distribution. We denote by $f_{\epsilon^+}$ such density. Recalling the general form of the L\'evy measure of a compound Poisson process we can write
\begin{align*}
f_{\epsilon^+}(z)=\frac{1}{\lambda^+_\epsilon}C\frac{e^{-Mz}}{z^{1+Y}}\mathtt{1}_{\{z>\epsilon\}} \; .
\end{align*}
It is possible to sample from such distribution by means of the acceptance rejection method as outlined in \cite{Cont2003}. Finally, we also introduce a diffusion approximation of small jumps. {The variance for the small jumps is given by \begin{equation*}
    \int_{-\epsilon}^\epsilon z^2 \nu(\de z) - \Big(\int_{-\epsilon}^\epsilon z \nu(\de z)\Big)^2\approx \int_{-\epsilon}^\epsilon z^2 \nu(\de z),
\end{equation*}
since the second term vanishes for small $\epsilon$, and hence, the diffusion coefficient is set to}
\begin{align*}
\sigma^2_\epsilon \coloneqq\int_{-\epsilon}^\epsilon z^2 \nu(\de z) \; .
\end{align*}
Again, due to the particular structure of the L\'evy measure, we can split the computation between the positive and negative jumps. For the positive jumps we have
\begin{align*}
\int_0^\epsilon z^2 C\frac{e^{-Mz}}{z^{1+Y}}\de z=\frac{C}{M^{2-Y}}\Gamma(a+1)\Gamma(a+1,M\epsilon) \; ,
\end{align*}
and similarly for the negative jumps. In summary, to simulate the CGMY process $L$, we introduce the discrete time approximation $L^\epsilon$ given by
\begin{align*}
L^\epsilon_{t_{n+1}}=L^\epsilon_{t_n}+\sigma_\epsilon \Delta W_{n+1}+\sum_{j=1}^{N_{n+1}^{+}} \Delta L_{n+1, j}^{+}-b_{\epsilon^+}{\Delta t}-\sum_{j=1}^{N_{n+1}^{-}} \Delta L_{n+1, j}^{-}+b_{\epsilon^-}{\Delta t} \; ,
\end{align*}
where
$N_{n+1}^{+}\sim \mathcal{P}\left(\lambda_{\epsilon}^{+}, \Delta t\color{black}\right)$, $N_{n+1}^{-}\sim \mathcal{P}\left(\lambda_{\epsilon}^{-},{\Delta t}\right)$, $\Delta L_{n+1, j}^{+},\sim f_{\epsilon^+}$, and 
$\Delta L_{n+1, j}^{-}\sim f_{\epsilon^-}$.
The approximating asset price process is then given by
\begin{equation}\label{eq:approxprice_true}
    X^\epsilon_{t_{n+1}} = X^\epsilon_{t_n} \exp\left\lbrace \Big({r-\frac{\sigma_\epsilon^2}{2}} -\int_{|z|\geq\epsilon}(e^z-1-z)\nu(\de z) \Big){\Delta t}+L^\epsilon_{t_{n+1}}-L^\epsilon_{t_{n}} \right\rbrace.\\
\end{equation}
In the above $\frac{\sigma_\epsilon^2}{2}$ and $\int_{|z|\geq\epsilon}(e^z-1-z)\nu(\de z)$ are the exponential compensators for the small jumps (modeled by a Brownian motion) and the large jumps (modeled by the Poisson random measure), respectively. On the other hand, for small $\epsilon$, $\frac{\sigma_\epsilon^2}{2} + \int_{|z|\geq\epsilon}(e^z-1-z)\nu(\de z)$ can be approximated by the full convexity correction for the CGMY model $\int_{\mathbb{R}}(e^z-1-z)\nu(\de z)$, for which we have the following expression \color{black}\begin{align*}
-\int_{\mathbb{R}}(e^z-1-z)\nu(\de z)=-C\Gamma(-Y)\left((M-1)^{Y}-M^Y+(G+1)^Y-G^Y\right).
\end{align*}
Therefore, we set
\color{black}
\begin{equation}\label{eq:approxprice}
    X^\epsilon_{t_{n+1}} = X^\epsilon_{t_n} \exp\left\lbrace \Big({r} -\int_{\mathbbm{R}}(e^z-1-z)\nu(\de z) \Big){\Delta t}+L^\epsilon_{t_{n+1}}-L^\epsilon_{t_{n}} \right\rbrace.\\
\end{equation}
We note that while equation \eqref{eq:approxprice_true} is a martingale, equation \eqref{eq:approxprice} is generally not. However, for the purposes of numerical implementation, this discrepancy has no significant impact as long as $\epsilon$ is small enough. A brief sensitivity analysis of $\epsilon$ is conducted in this section.
\color{black}
\subsubsection{Call option under the CGMY model}
Let $X$ be the price of a stock described by \eqref{eq:approxprice}
and $Y$ a European call option with value
\begin{equation}
Y_t = \mathbb{E}^\mathbb{Q}\left[e^{-r(T-t)}(X_T-k)^+ |\cF_t \right].
\end{equation}

Our goal is to estimate the value $Y_0$ using the deep solver and compare it with a standard Monte Carlo simulations. 
Before we proceed with the estimation, we test the goodness of fit of the proposed approximation by comparing it with the semi closed-form density obtained via the Fast Fourier Transform (FFT) applied to the characteristic function of the CGMY process. In Figure \ref{fig:density_plot} we are able to verify the goodness of fit.
\begin{figure}[htbp]
\centering
\includegraphics[scale=0.7]{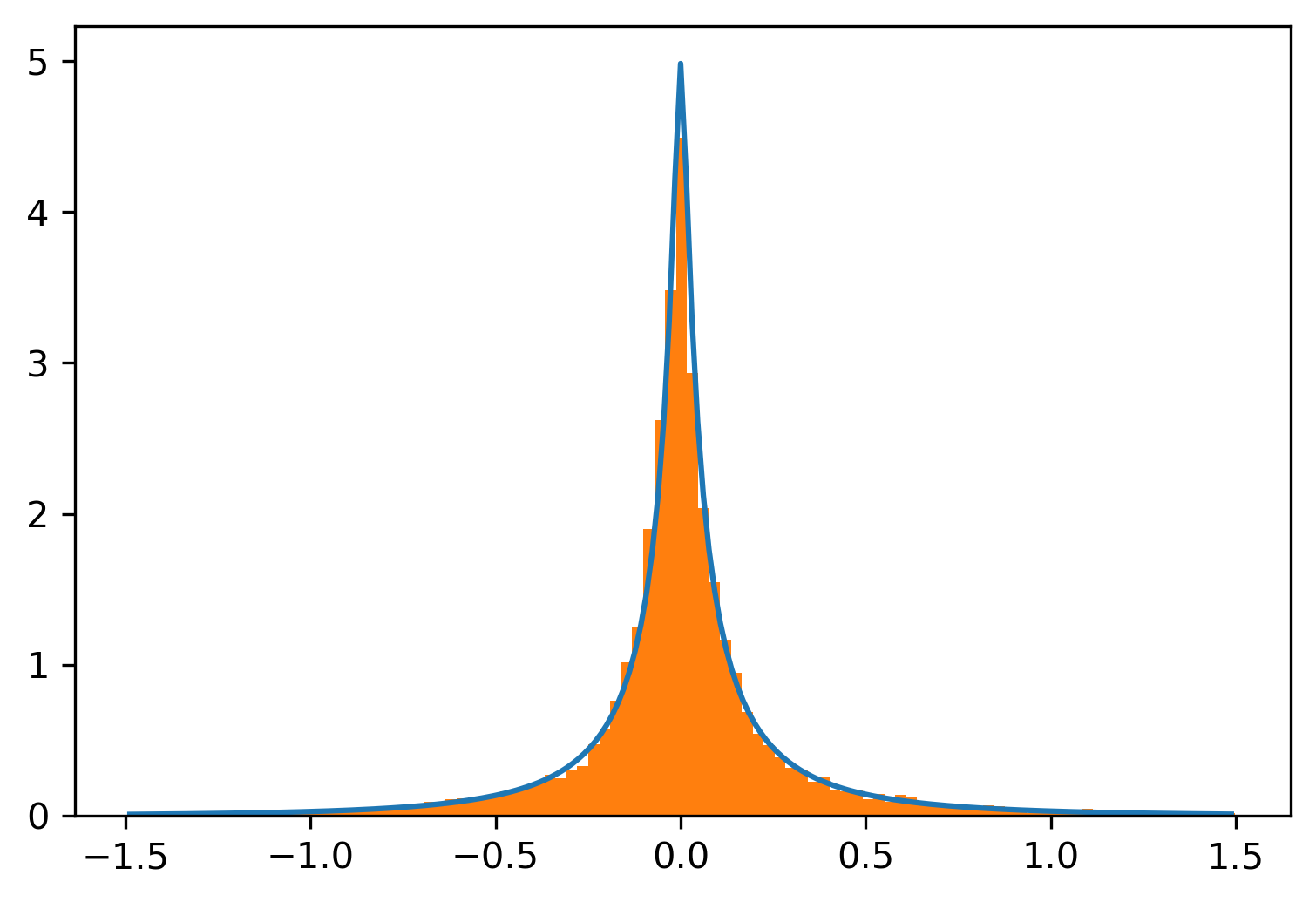}
\caption{Histogram of empirical density in orange and true density in blue.}
\label{fig:density_plot}
\end{figure}

We set $X_0=1$, $k=0.9$, $r=0.04$, $C=0.1$, $G=13$, $M=14$, $Y=0.6$, $\epsilon=0.00001$. We consider $M_\text{steps}=100$ time steps  (to allow $M$ to represent the parameter in the CGMY process, as is customary, we introduce a new notation for the number of time steps in this section), $10,000$ batch iterations with batch size $2^{8}=256$. We use the fast Fourier transform (FFT) to compute reference solutions along each path to obtain pathwise option values. In Figure \ref{fig:CallRes2}, to the right, five representative samples obtained from the deep solver are compared to the reference solutions. \color{black}

\begin{figure}[htp]
\centering
\begin{tabular}{cc}
        \includegraphics[width=70mm]{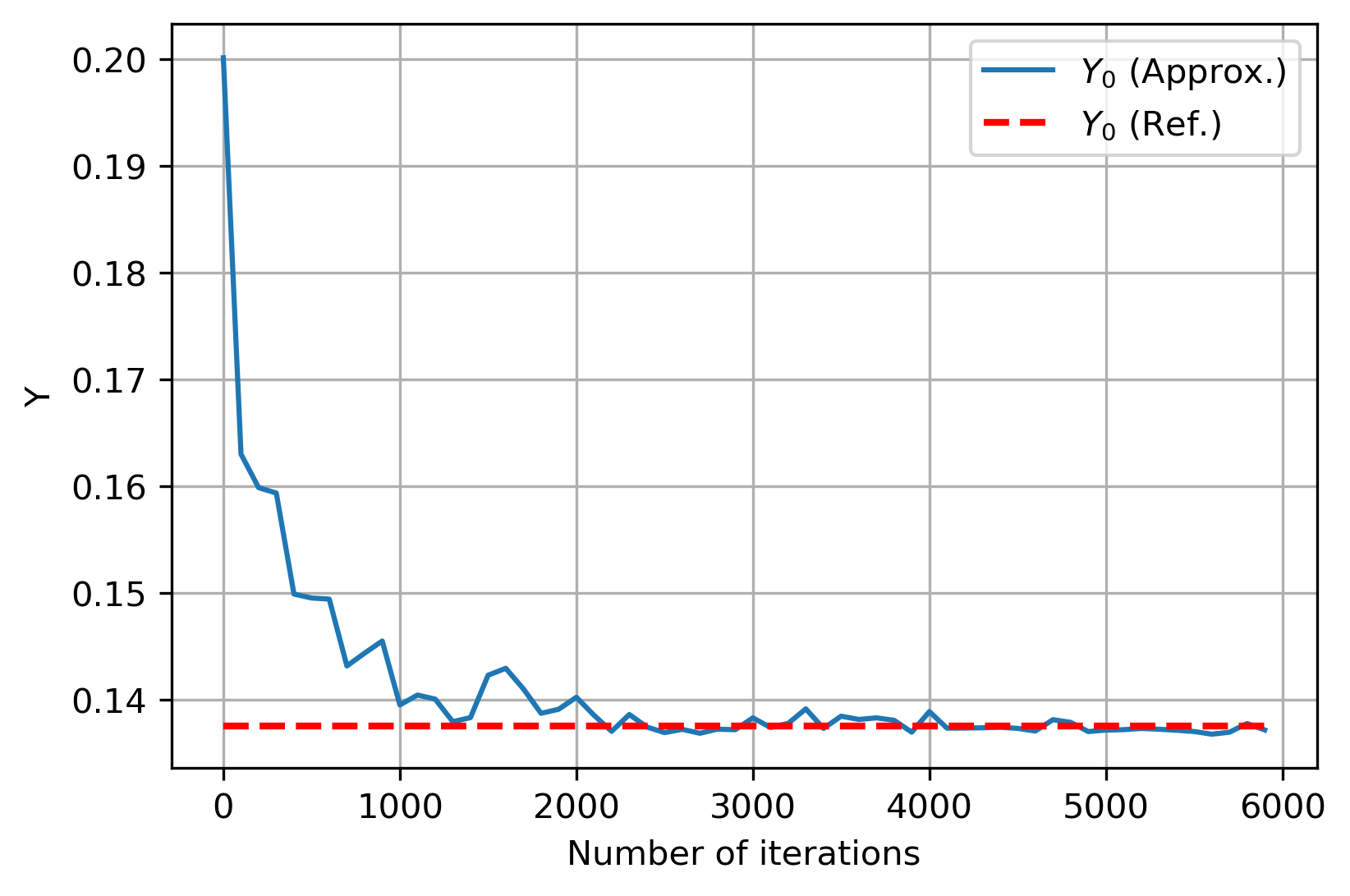}& 
          \includegraphics[width=70mm]{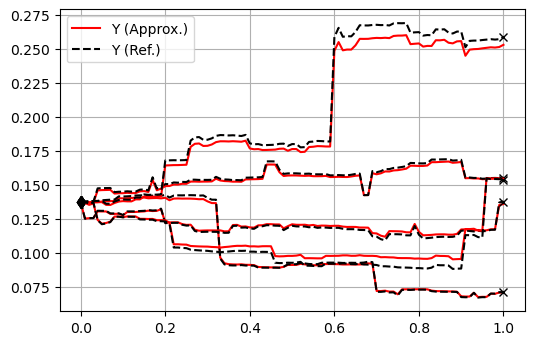}
\end{tabular}
\caption{{\textbf{Left:} Approximate initial value of the BSDE against the number of training steps (the number of batch iterations). \textbf{Right:} Five representative paths of our approximations compared to reference solutions, computed with the semi-analytic solution obtained by the FFT along each path.}}\label{fig:CallRes2}
\end{figure}
{Finally, we analyze the performance of the algorithm for various values of $\epsilon$. The results, shown in Table \ref{tab:epsilon}, indicate significant approximation  errors for $\epsilon>0.0001$. Additionally, we observe that computational time increases as $\epsilon$ decreases, due to the need for more jumps, which are computationally more demanding to simulate than the approximating Brownian motion. It should be noted, however, that we have not optimized the performance of this aspect of the algorithm.} 

\begin{table}[htbp]	
\caption{{Results for different values of $\epsilon$. Reference value for $Y_0$ computed with the FFT (with $\epsilon=0$) is given by $0.1375$.}}
\centering
\begin{tabular}{cccc}
\toprule
$\epsilon$ &	$\widetilde{Y}_0^\Theta\ (Y_0=1375)$	& Relative error & Time (s)  \\
\midrule\midrule
0.005	&	0.295	&	115\%	&	1698		\\
0.001	&	0.149	&	8\%	&	1778			\\
0.0005	&	0.155	&	13\%	&	1932			\\
0.0001	&	0.134	&	2.4\%	&	2329			\\
0.00005	&	0.138	&	0.7\%	&	3526			\\
0.00001	 &	0.137	&	0.4\%	&	9332			\\
\bottomrule
\end{tabular}
\label{tab:epsilon}
\end{table}

\subsection{Application to counterparty credit risk}
{In the examples above, we considered the pricing of some financial payoffs. It would be natural to argue that, in pricing applications, a simple use of a Monte Carlo method allows us to approximate the price at the initial time. This is indeed the way we use to benchmark our results when closed formulas are not available. However, the knowledge of the initial price of a contract is in most cases insufficient in view of managing the risks affecting the claim: typically one needs the possibility to simulate future scenarios of the price process. The Deep BSDE solver with jumps we propose allows us to perform this task: once the training of the family of ANNs has been performed, we can simulate the BSDE's trajectories, thus producing future simulated scenarios for the price process. 
To stress the importance of this feature of the algorithm, let us now focus on the issue of counterparty credit risk. In the aftermath of the financial crisis, the market practice for the valuation of derivative contracts has changed significantly. It is now standard to adjust the standard risk neutral price in order to account for a series of frictions. For the sake of brevity, we concentrate here on the possibility that the bank's conterparty defaults and refer to \cite{bgo2019} for a complete treatment. \\On a filtered probability space as before,  let $Y = (Y_t)_{t\in [0,T]}$ denote the value of a contingent claim with a $\mathcal F_T$-measurable payoff $H$ depending on some jump-diffusion process of the form \eqref{eq:forward}. Let $r=(r_t)_{t\in [0,T]}$ be the discount rate, then the standard risk-neutral valuation approach leads to the following expression for the price of the claim in the absence of counterparty risk
$$
Y_t = \mathbb E^{\mathbb Q} \left[e^{-\int^T_t r_s \de s} H \big| \mathcal F_t  \right]
$$ 
which we refer to as the clean value. The counterparty, however, could default at some stopping time $\tau_C$ with hazard rate $\lambda^C = (\lambda^C_t)_{t\in [0,T]}$. To account for this potential loss the Credit Value Adjustment (CVA) is introduced. The value of the contract is reduced by the expected loss represented by the CVA, so that the adjusted value of the contract $V = (V_t)_{t\in [0,T]}$ is given by 
$$
V_t = Y_t - \text{CVA}_t,
$$ 
where the CVA is given by the following conditional expectation
$$
 \text{CVA}_t = \mathbb E^{\mathbb Q} \left[ \int^T_t \lambda^C_u e^{-\int^u_t (r_s + \lambda^C_s) \de s} \left(Y_u\right)^+  \de u\Big| \mathcal F_t\right].
$$
Inside the conditional expectation one can notice the presence of the  positive part of $Y_u$, $u\in [t,T]$, that represents the potential future liability of the counterparty towards the bank. It is then clear from the CVA formula above that to compute the adjusted price of the claim we need the ability to simulate future scenarios of the clean value $Y$ of the contract. As $Y$ satisfies a BSDE,  our algorithm can be employed to solve such problem. For illustrative purposes we plot in Figure \ref{fig:price_paths} paths of the underlying risk factor using the model from Section \ref{sect:calloption} together with the corresponding sample paths of the call option. 

An in depth analysis of the present algorithm for the computation of the whole set of  valuation adjustments (xVA) can be obtained along the lines of \cite{gpr2020}.

\begin{figure}[htbp]
\centering
\includegraphics[scale=0.55]{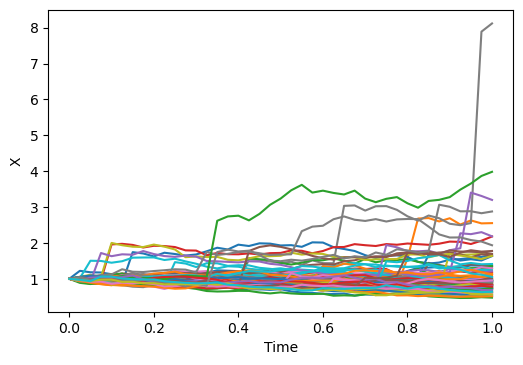} 
\includegraphics[scale=0.55]{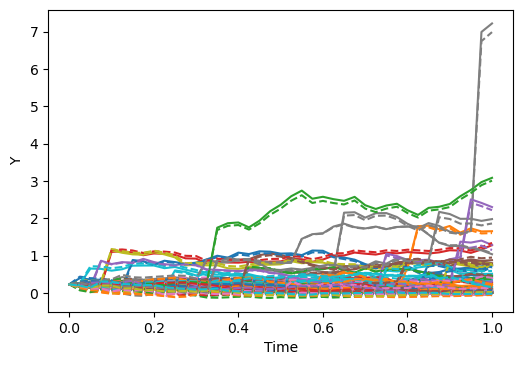}
\caption{{\textbf{Left:} Simulated paths of the forward SDE for the stock price. \textbf{Right: } Simulated paths of the BSDE. Analytical values in solid lines and approximate values in dashed lines. The model under consideration is the Merton Jump diffusion employed in Section \ref{sect:calloption}}}
\label{fig:price_paths}
\end{figure}

}
\section{Conclusions}
In this paper, we presented an extension of the deep BSDE solver of \cite{ehanjen17} and \cite{Han2018} to the setting of FBSDEs with jumps. As in the original formulation of the solver, we discretize the FBSDE and parametrize it by means of ANNs. A delicate part of the proposed extension is given by the treatment of the integral with respect to the compensator inside the jump term, which involves an integration over the whole support of the L\'evy measure. The approximation of such integral terms is increasingly difficult as the dimension of the space increases. With the aim of formulating an algorithm which is feasible in a possibly high-dimensional setting, we parametrize both the integral with respect to the Poisson random measure and the integral with respect to the jump compensator via two further families of ANNs. The parameters of such additional ANNs are then optimized by introducing a second penalty term in the loss that enforces the martingale property of the compensated jump integrals at every point in time.

We test the proposed methodology first in a one-dimensional setting, which allows us to validate the algorithm against known alternative approaches for the computation of the solution. After that, we introduce a high dimensional example, i.e. the PIDE associated to a basket option on $100$ underlyings as a proof for the fact that the algorithm is feasible in high dimension. Finally we provide an example that shows how one can treat the case where the jumps are of infinite activity.

There are several directions for future research. First of all, even though a full theory for the approximation quality of ANNs is still missing, one could still provide a convergence analysis of the algorithm in the spirit of \cite{HanLon18}. A further line of research could look at applications in finance: in \cite{gpr2020} the authors employ the deep BSDE solver to evaluate valuation adjustments in the context of counterparty credit risk in a pure diffusive setting. Since jump processes have not enjoyed a large popularity in the context of counterparty credit risk mainly due to the increased numerical complexity, our approach could allow the use of more complex and realistic asset price models in such context. Another stream of application could be given by quadratic hedging problems, where the deep solver can also be employed as in \cite{glp2020}.
\color{black}

\section*{Acknowledgements} We are grateful to Antonis Papapantoleon, Roxana Dumitrescu and two anonymous referees for useful remarks. 
{Moreover, we would like to express our gratitude to all the participants of the "Quantitative Finance Workshop" held in Gaeta on April 2023, the "SIAM conference on Financial Mathematics and Engineering" held in Philadelphia on July 2023, and the "World Online Seminars on Machine Learning in Finance" held on October 2022, for the insightful discussions and constructive feedbacks we received on this work.}

\bibliographystyle{apalike}
\bibliography{biblio}

\appendix

\end{document}